\documentclass[11pt]{article}

\usepackage[latin1]{inputenc} 
\usepackage{amsthm,amsmath,amssymb,bm} 
\usepackage{graphicx}
\usepackage{color}
\usepackage{verbatim}
\usepackage{float}

\usepackage{array}   
\newcolumntype{C}{>{$}c<{$}}
\newcolumntype{L}{>{$}l<{$}}

\usepackage{enumitem} 
\setlist[enumerate]{topsep=0pt,itemsep=-1ex,partopsep=1ex,parsep=1ex}
\setlist[itemize]{topsep=0pt,itemsep=-1ex,partopsep=1ex,parsep=1ex}

\theoremstyle{plain}
\newtheorem{theo}{Theorem}[section]

\newtheorem{lemma}[theo]{Lemma}

\theoremstyle{definition}
\newtheorem{defn}[theo]{Definition}

\newtheorem{set}[theo]{Setting}

\newcommand{\mc}[1]{\mathcal{#1}}
\newcommand{\mb}[1]{\mathbb{#1}}
\newcommand{\mf}[1]{\mathfrak{#1}}
\newcommand{\nib}[1]{\noindent {\bf #1}}

\newcommand{\bracc}[1]{\left\{ #1 \right\}}
\newcommand{\bsize}[1]{\left| #1 \right|}
\newcommand{\brak}[1]{\left[ #1 \right]}

\newcommand{\sgen}[1]{\langle #1 \rangle}

\newcommand{\sub}{\subseteq}

\newcommand{\Lra}{\Leftrightarrow}

\newcommand{\sm}{\setminus}
\newcommand{\ov}{\overline}
\newcommand{\ul}{\underline}
\newcommand{\wt}{\widetilde}
\newcommand{\eps}{\varepsilon}
\newcommand{\es}{\emptyset}

\newcommand{\ova}{\overrightarrow}
\newcommand{\lova}{\overleftarrow}

\newcommand{\aA}{\alpha}
\newcommand{\bB}{\beta}

\newcommand{\dD}{\delta}

\newcommand{\tT}{\theta}

\newcommand{\oO}{\omega}

\newcommand{\sd}{\bigtriangleup}

\newcommand{\DD}{\Delta}
\newcommand{\OO}{\Omega}

\newcommand{\Ss}{\Sigma}

\newcommand{\wK}{\ova{W}^K_{\! 8}}
\newcommand{\wKc}{\ova{W}^K_{\! c}}
\newcommand{\wL}{\ova{W}_{\! \ell^*}}
\newcommand{\wC}{\ova{W}_{\! c}}

\title{The generalised Oberwolfach problem}
\author{Peter Keevash\thanks{Mathematical Institute,
University of Oxford, Oxford, UK. Email: keevash@maths.ox.ac.uk.}
\and Katherine Staden\thanks{Mathematical Institute,
University of Oxford, Oxford, UK. Email: staden@maths.ox.ac.uk.
\newline \hspace*{1.8em}Research supported
in part by ERC Consolidator Grant 647678.}}

\begin{document}

\maketitle

\begin{abstract}
We prove that any quasirandom dense large graph
in which all degrees are equal and even can be 
decomposed into any given collection of two-factors
($2$-regular spanning subgraphs).
A special case of this result gives a new solution to 
the Oberwolfach problem.
\end{abstract}

\section{Introduction}

At meals in the Oberwolfach Mathematical Institute,
the participants are seated at circular tables.
At an Oberwolfach meeting in 1967, Ringel (see \cite{LR})
asked whether there must exist a sequence of seating plans
so that every pair of participants 
sit next to each other exactly once.
We assume, of course, that there are an odd number of participants,
as each participant sits next to two others in each meal.
The tables may have various sizes,
which we assume are the same at each meal.

\medskip

\nib{Oberwolfach Problem (Ringel).}
Let $F$ be any two-factor (i.e.\ $2$-regular graph) 
on $n$ vertices, where $n$ is odd.
Can the complete graph $K_n$ be
decomposed into copies of $F$?

\medskip

We obtain a new solution of this problem for large $n$, 
with a theorem that is more general in three respects:
(a) we can decompose any dense quasirandom graph
that is regular of even degree
(not just $K_n$ for $n$ odd),
(b) we can decompose into any prescribed 
collection of two-factors
(not just copies of some fixed two-factor $F$),
(c) our theorem applies to directed graphs (digraphs).

We start by stating our result for undirected graphs.
We require the following quasirandomness definition.
We say that a graph $G$ on $n$ vertices is
$(\eps,t)$-typical if every set $S$ of at most $t$ vertices 
has $((1  \pm \eps)d(G))^{|S|} n$ common neighbours, where 
$d(G) = e(G) \tbinom{n}{2}^{-1}$ is the density of $G$.

\begin{theo}\label{mainundir}
For all $\aA>0$ there exist $t,\eps,n_0$ such that 
any $(\eps,t)$-typical graph on $n \ge n_0$ vertices
that is $2r$-regular for some integer $r > \aA n$
can be decomposed into any family of $r$ two-factors.
\end{theo}

Theorem \ref{mainundir} implies some variant forms of the
Oberwolfach problem that have appeared in the literature,
such as the Hamilton--Waterloo Problem (two types of two-factors),
or that if $n$ is even then $K_n$ can be decomposed into a
perfect matching and any specified collection 
of $n/2-1$ two-factors. More generally, 
with parameters as in Theorem \ref{mainundir},
it is easy to deduce that
any $(\eps,t)$-typical graph on $n \ge n_0$ vertices
that is $(2r+1)$-regular for some integer $r > \aA n$
can be decomposed into a perfect matching
and any family of $r$ two-factors.

We will deduce Theorem~\ref{mainundir}
from the directed version below.
First we extend our definitions to digraphs.
We say that a digraph $G$ on $n$ vertices is
$(\eps,t)$-typical if for every set $S=S^- \cup S^+$ 
of at most $t$ vertices 
there are $((1  \pm \eps)d(G))^{|S|} n$ vertices which
are both common inneighbours of $S^-$
and outneighbours of $S^+$, where 
$d(G) = e(G) \tbinom{n}{2}^{-1}$ is the density of $G$.
We say that $G$ is $r$-regular if 
$d^+_G(v)=d^-_G(v)=r$ for all $v \in V(G)$.
A \emph{one-factor} is a $1$-regular digraph;
equivalently, it is a union of 
vertex-disjoint oriented cycles.

\begin{theo}\label{main}
For all $\aA>0$ there exist $t,\eps,n_0$ such that 
any $(\eps,t)$-typical digraph on $n \ge n_0$ vertices
that is $r$-regular for some integer $r > \aA n$
can be decomposed into any family of $r$ one-factors.
\end{theo}

Theorem~\ref{mainundir} follows from Theorem~\ref{main}
and the observation that for any typical graph
that is regular of even degree
there exists an orientation 
which is a regular typical digraph.
To see this, one can orient edges independently at random
and make a few modifications to obtain the required orientation.
(See Lemma~\ref{typ:split} below for a similar argument.)

While we were preparing this paper,
the Oberwolfach problem (for large $n$) was solved by
Glock, Joos, Kim, K\"uhn and Osthus~\cite{GJKKO}. 
They also obtained a more general result that covers
the other undirected applications just mentioned,
but our result is more general than theirs 
in the three respects mentioned above:
(a) we can decompose any dense typical regular graph
(whereas their result only applies to almost complete graphs),
(b) we can decompose into any collection of two-factors
(whereas they can allow for a collection of two-factors
provided that some fixed $F$ occurs $\OO(n)$ times),
(c) our result also applies to digraphs
(whereas theirs is for undirected graphs).

There is a large literature on the Oberwolfach Problem,
of which we mention just a few highlights
(a more detailed history is given in \cite{GJKKO}). 
The problem was solved for infinitely many $n$
by Bryant and Scharaschkin \cite{BS},
in the case when $F$ consists of two cycles by Traetta \cite{T},
and for cycles of equal length by
Alspach, Schellenberg, Stinson and Wagner \cite{ASSW}.
A related conjecture of Alspach that $K_n$ can be
decomposed into any collection of cycles each of length $\le n$
and total size $\tbinom{n}{2}$ was solved
by Bryant, Horsley and Pettersson \cite{BHP}. 

There are several recent general results on approximate
decompositions that imply an approximate solution
to the generalised Oberwolfach Problem, 
i.e.\ that any given collection of two-factors
can be embedded in a quasirandom graph
provided that a small fraction of the edges
can be left uncovered: we refer to the papers of
Allen, B\"ottcher, Hladk\'y and Piguet \cite{ABHP},
Ferber, Lee and Mousset \cite{FLM}
and Kim, K\"uhn, Osthus and Tyomkyn \cite{KKOT}.

\medskip

\nib{Notation.}

Given a graph $G = (V,E)$, 
when the underlying vertex set $V$ is clear, 
we will also write $G$ for the set of edges. 
So $|G|$ is the number of edges of $G$. Usually $|V|=n$.
The \emph{edge density} $d(G)$ of $G$ is $|G|/\tbinom{n}{2}$.
We write $N_G(x)$ for the neighbourhood of a vertex $x$ in $G$.
The degree of $x$ in $G$ is $d_G(x)=|N_G(x)|$.
For $A \subseteq V(G)$, we write 
$N_G(A) := \bigcap_{x \in A}N_G(x)$;
note that this is the common neighbourhood 
of all vertices in $A$, not the neighbourhood of $A$.

In a directed graph $J$ with $x \in V(J)$, 
we write $N_J^+(x)$ for the set of out-neighbours of $x$ in $G$ 
and $N_G^-(x)$ for the set of in-neighbours.
We let $d^\pm_G(A) := |N^\pm_G(A)|$.
We define common out/in-neighbourhoods 
$N_J^\pm(A) = \bigcap_{x \in A} N_J^\pm (A).$

We say $G$ is \emph{$(\eps,t)$-typical} 
if $d_G(S) = ((1  \pm \eps)d(G))^{|S|} n$ 
for all $S \subseteq V(G)$ with $|S| \le t$. 

We say that an event $E$ holds with high probability (whp) 
if $\mb{P}(E) > 1 - \exp(-n^c)$ for some $c>0$ and $n>n_0(c)$.
We note that by a union bound for any fixed collection $\mc{E}$
of such events with $|\mc{E}|$ of polynomial growth  
whp all $E \in \mc{E}$ hold simultaneously. 

We omit floor and ceiling signs for clarity of exposition.

We write $a \ll b$ to mean $\forall\ b>0 \
\exists\ a_0>0 \ \forall\ 0<a<a_0$.

We write $a \pm b$ for an unspecified number in $[a-b,a+b]$.

Throughout the vertex set $V$ will come with a cyclic order,
which we usually identify with the natural cyclic order
on $[n]=\{1,\dots,n\}$. For any $x \in V$ we write $x^+$
for the successor of $x$, so if $x \in [n]$ then 
$x^+$ is $x+1$ if $x \ne n$ or $1$ if $x=n$.
We define the predecessor $x^-$ similarly. Given $x,y$ in $[n]$
we write $d(x,y)$ for their cyclic distance,
i.e.\ $d(x,y) = \min \{ |x-y|, n-|x-y| \}$.

\section{Overview of the proof} \label{sec:over}

We will illustrate the ideas of our proof by starting
with a special case and becoming gradually more general.
Suppose first that we wish to decompose a typical dense (undirected)
$2r$-regular graph $G$ on $n$ vertices into $r$ triangle-factors
(i.e.\ two-factors in which each cycle is a triangle
-- we require $3 \mid n$ for this question to make sense).
The existence of such a decomposition 
(also known as a resolvable triangle-decomposition of $G$)
follows from a recent result of the first author \cite{K2}
generalising the existence of designs (see \cite{Kexist})
to many other `design-like' problems. The proof in \cite{K2} goes 
via the following auxiliary decomposition problem, 
which also plays an important role in this paper.

Let $J$ be an auxiliary graph with $V(J)$ partitioned 
as $V \cup W$, where $V=V(G)$ and $|W|=r$.
Let $J[V]=G$, $J[V,W]=V \times W$ and $J[W]=\es$.
Note that a decomposition of $G$ into triangle-factors
is equivalent to a decomposition of $J$ into copies of $K_4$
each having $3$ vertices in $V$ and $1$ vertex in $W$.
Indeed, given such a decomposition of $J$, for each $w \in W$
we define a triangle-factor of $G$ by removing $w$ from all
copies of $K_4$ containing $w$ in the decomposition;
clearly every edge of $G$ appears in exactly one 
of these triangle-factors. Conversely, any decomposition of $G$
into triangle-factors can be converted into a suitable
$K_4$-decomposition of $J$ by adding each $w \in W$ to one of 
the triangle-factors (according to an arbitrary matching).

The auxiliary construction described above is quite flexible,
so a similar argument covers many other cases of our problem.
For example, decomposing $G$ into $C_\ell$-factors
(two-factors in which each cycle has length $\ell$)
is equivalent to decomposing $J$ into `wheels' $W_\ell$
with `rim' in $V$ and `hub' in $W$. (We obtain $W_\ell$ from
$C_\ell$, which is called the rim, by adding a new vertex, 
called the hub, joined to every other vertex, 
by edges that we call spokes.)
Such a decomposition exists by \cite{K2}.

We can encode our generalised Oberwolfach Problem
in full generality by introducing colours on the edges.
For each possible cycle length $\ell$ we introduce a colour,
which we also call $\ell$. For each $w \in W$, we denote its
corresponding factor by $F_w$, and suppose that it has $n^w_\ell$
cycles of length $\ell$ (where $\sum_\ell \ell n^w_\ell = n$).
We colour $J$ so that each $w \in W$ is incident to exactly 
$n^w_\ell$ edges of colour $\ell$, and all other edges are uncoloured.
We colour each $W_\ell$ so that exactly one spoke has colour $\ell$
and all other edges are uncoloured. Then a decomposition of $G$ into
$\{F_w: w \in W\}$ is equivalent to a decomposition of $J$
into wheels with this colouring with rim in $V$ and hub in $W$.
Note that this equivalence does not depend on which edges of $J$
we colour, but to apply \cite{K2} we will require the colouring 
to be suitably quasirandom. Another important constraint
in applying \cite{K2} is that the number of colours
and the size of the wheels should be bounded by an absolute constant.
Thus our generalised Oberwolfach Problem can only be solved by
direct reduction to \cite{K2} in the case that all factors
have all cycle lengths bounded by some absolute constant.

This now brings us to the crucial issue for this paper:
how can we encode two-factors with cycles of arbitrary length
by an auxiliary construction to which \cite{K2} applies?
Before describing this, we pass to an auxiliary problem
of decomposing a subgraph $G'$ of $G$ into graphs
$(G_w: w \in W)$, where each $G_w$ is a vertex-disjoint
union of paths with prescribed endpoints, lengths and vertex set.
More precisely, for each $w \in W$ we are given specified 
lengths $(\ell^w_i: i \in I_w)$, vertex-pairs 
$((x^w_i,y^w_i): i \in I_w)$, a forbidden set $Z_w$, 
and we want each $G_w$ to be a union of vertex-disjoint
$x^w_i y^w_i$-paths of length $\ell^w_i$ for each $i \in I_w$
with $V(G_w)=V(G) \sm Z_w$. We will arrive at this problem
having embedded some subgraphs $F'_w \sub F_w$ of each $w \in W$, 
so the prescribed endpoints will be endpoints of
paths in $F'_w$ that need to be connected up to form cycles,
and $Z_w$ will consist of all vertices of degree $2$ in $F'_w$.
We assume that all lengths $\ell^w_i$ are divisible by $8$
(which is easy to ensure for long cycles).

We will translate the above path factor problem 
into an equivalent problem of decomposing a certain
auxiliary two-coloured directed graph $J$,
with $V(J) = V \cup W$ as in the previous construction.
We call the two colours `$0$' (which means `uncoloured')
and `$K$' (which means `special'). Again, $J[W]=\es$.
For now we defer discussion of $J[V,W]$ 
and describe the arcs of $J[V]$,
which are in bijection with the edges of $G$.
For colour $0$ this bijection simply corresponds
to a choice of orientation for edges, but for 
colour $K$ we employ the following `twisting' construction.
We fix throughout a cyclic order of $V$,
and require that each arc $\ova{xy}$ of colour $K$ in $J$
comes from an edge $xy^+$ of $G$, where $y^+$ denotes 
the successor of $y$ in the cyclic order.

Consider any directed $8$-cycle $C$ in $J$ with vertex sequence
$x_1 \dots x_8$, such that all arcs have colour $0$ 
except that $\ova{x_7 x_8}$ has colour $K$.
The edges in $G$ corresponding to $C$ form a path
with vertex sequence $x_8 x_1 \dots x_7 x_8^+$.
Now suppose we have a family of such cycles
$\mc{C} = (C^i: i \in I)$ where each $C^i$
has vertex sequence $x^i_1 \dots x^i_8$.
Call $\mc{C}$ compatible if 
(i) its cycles are mutually vertex-disjoint, and
(ii) if any $(x^i_8)^+$ is used by a cycle in
$\mc{C}$ then it is some $x^j_8$.
Suppose $\mc{C}$ is compatible and let
$([x_j,y_j]: j \in J)$ denote the family of maximal 
cyclic intervals contained in $\{x^i_8: i \in I\}$.
Then the edges of $G$ corresponding to the cycles of $\mc{C}$
form a family of vertex-disjoint paths $(P_j: j \in J)$, where 
each $P_j$ is an $x_j y_j^+$-path whose vertex sequence is the 
concatenation of vertex sequences of the $8$-paths as described 
above for each cycle of $\mc{C}$ using a vertex of $[x_j,y_j]$.
 
\begin{figure}
\hspace*{-0.4cm}\includegraphics[scale=0.85]{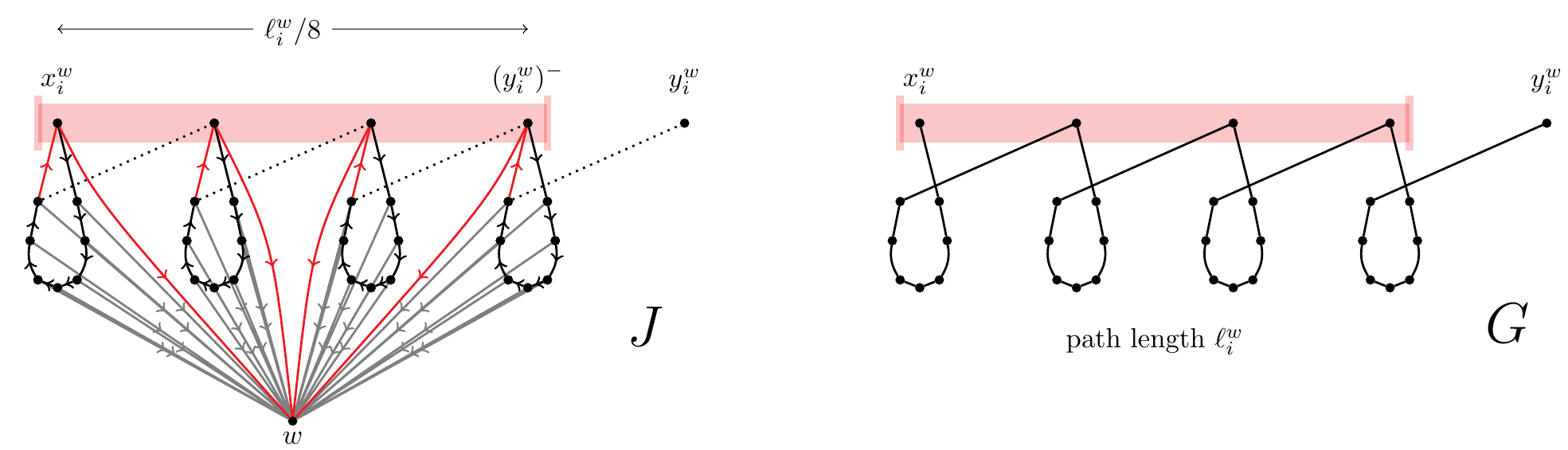}
\end{figure}

The above construction allows us to pass from the path factor problem
to finding certain edge-disjoint compatible cycle families in $J$.
In order for our path factor problem to obey the constraints of
this encoding we require the prescribed vertex-pairs for each $w$
to define disjoint cyclic intervals $([x^w_i,(y^w_i)^-]: i \in I_w)$ 
of lengths $\ell^w_i/8$ (and also that no successor $y^w_i$
is contained in any of the other intervals for $w$).
We are thus introducing extra constraints into the path factor 
problem that may affect up to $n/8$ vertices for each $w$,
but the flexibility on the remaining
vertices will be sufficient.

Now we can complete the description of the auxiliary graph $J$
and the decomposition problem that encodes the compatible cycle
family problem. We define $J[V]$ as above,
and $J[V,W]$ so that all arcs are directed
towards $W$, each in-neighbourhood $N^-_J(w)$ is obtained
from $V(G) \sm Z_w$ by deleting the interval successors
$\{ y^w_i: i \in I_w \}$, all arcs $\ova{xw}$
with $x$ in an interval $[x^w_i,(y^w_i)^-]$ are coloured $K$,
and all other arcs of $J[V,W]$ are coloured $0$.
Finally, the compatible cycle family problem is equivalent
to decomposing $J$ into coloured directed wheels $\wK$,
obtained from $W_8$ by directing the rim cyclically, 
directing all spokes towards the hub $w$, giving colour $K$
to one rim edge $\ova{xy}$ and one spoke $\ova{yw}$,
and colouring the other edges by $0$. 
The deduction from \cite{K2} of the existence of wheel 
decompositions is given in section \ref{sec:wheel}.

We now describe the strategy for the proof of Theorem \ref{main}.
The goal is to embed some parts of our two-factors
so that the remaining problem is of one of two special types
that has an encoding suitable for applying \cite{K2},
either a path factor problem 
encoded as $\wK$-decomposition
or a $C_\ell$-factor problem 
encoded as $\ova{W}_{\!\ell}$-decomposition
(we take the coloured wheel $W_\ell$ discussed above
for $C_\ell$-factors and introduce directions 
as in $\wK$, which are not necessary
but convenient for giving a unified analysis).
We call a factor `long' if it has at least $n/2$ vertices
in cycles of length at least $K$ (as well as denoting the
special colour, $K$ is also used as a large constant length
threshold, above which we treat cycles using the special
twisting encoding as above). We call the other factors `short'.

We start by reducing to the case that all factors are long
or all factors are short. To do so, suppose first that
there are $\OO(n)$ long factors and $\OO(n)$ short factors.
Then we can randomly partition $G$ into typical graphs 
$G^L$ and $G^S$, each of which is regular of the correct degree
(twice the number of long factors for $G^L$
and twice the number of short factors for $G^S$).
If there are $o(n)$ factors of either type then these can 
be embedded one-by-one (by the blow-up lemma \cite{KSS}),
and then the remaining problem still satisfies the conditions
of Theorem \ref{main} (with slightly weaker typicality).
The short factor problem can be further reduced to the case
that there is some length $\ell^*$ such that each factor has
$\OO(n)$ cycles of length $\ell^*$. Indeed, we can divide the
factors into a constant number of groups according to some
choice of cycle length that appears $\OO(n)$ times in each factor
of the group. Any group of $o(n)$ factors can be embedded greedily,
so after taking a suitable random partition, it suffices 
to show that the remaining groups can each be embedded 
in a graph that is typical and regular of the correct degree.

Thus we can assume that we are in one of the following cases.
Case $K$: all factors are long,
our goal is to reduce to $\wK$-decomposition.
Case $\ell^*$: all factors have $\OO(n)$ cycles of length $\ell^*$,
our goal is to reduce to $\wL$-decomposition. 
In any case, the reduction is achieved by applying an approximate decomposition result in a suitable random subgraph,
in which we embed a subgraph of each of our factors.
At this step, in Case $\ell^*$ we embed
all cycles of length $\ne \ell^*$,
and in Case $K$ we embed all short cycles
and some parts of the long cycles as needed 
to reduce to a suitable path factor problem.

This approximate decomposition result is superficially 
similar to the maximum degree $2$ case of
the blow-up lemma for approximate decompositions
due to Kim, K\"uhn, Osthus and Tyomkyn \cite{KKOT}.
However, it does not suffice to use their result,
as we require a decomposition that is compatible
with the conditions of our final decomposition problem 
(into $\wK$ or $\wL$),
so the sets of vertices of the partial factors  
embedded in this step must be suitably quasirandom
and avoid the intervals needed for Case $K$.
Furthermore, we obtain the required approximate
decomposition by similar arguments to those
for the exact decomposition, 
which does not add much extra work.

The technical heart of the paper is a randomised algorithm 
(presented in section \ref{sec:alg}), which gives 
a unified treatment of the cases described above.
It simultaneously (a) partitions almost all of $G$ 
into two graphs $G_1$ and $G_2$, and
(b) sets up auxiliary digraphs $J_1$ and $J_2$ such that 
(i) an approximate wheel decomposition of $J_2$
gives an approximate decomposition of $G_2$
into the partial factors described above, and
(ii) the graph $G'_1$ of edges that are unused by
the approximate decomposition has an auxiliary digraph
that is a sufficiently small perturbation of $J_1$ that 
it can still be used for the exact decomposition step.
The analysis of the algorithm falls naturally into two parts:
the choice of intervals (section \ref{sec:int}),
then regularity properties of an auxiliary 
hypergraph defined by wheels (section \ref{sec:reg}).
The results of this analysis are applied 
to show the existence of the various partial factor
decompositions discussed above:
the approximate step is in section \ref{sec:approx} 
and the exact step in section \ref{sec:exact}.
Section \ref{sec:pf} combines
all the ingredients prepared in the previous sections
to produce the proof of our main theorem.
The final section contains some concluding remarks.

\begin{figure}[H]
\centering
\includegraphics[scale=0.8]{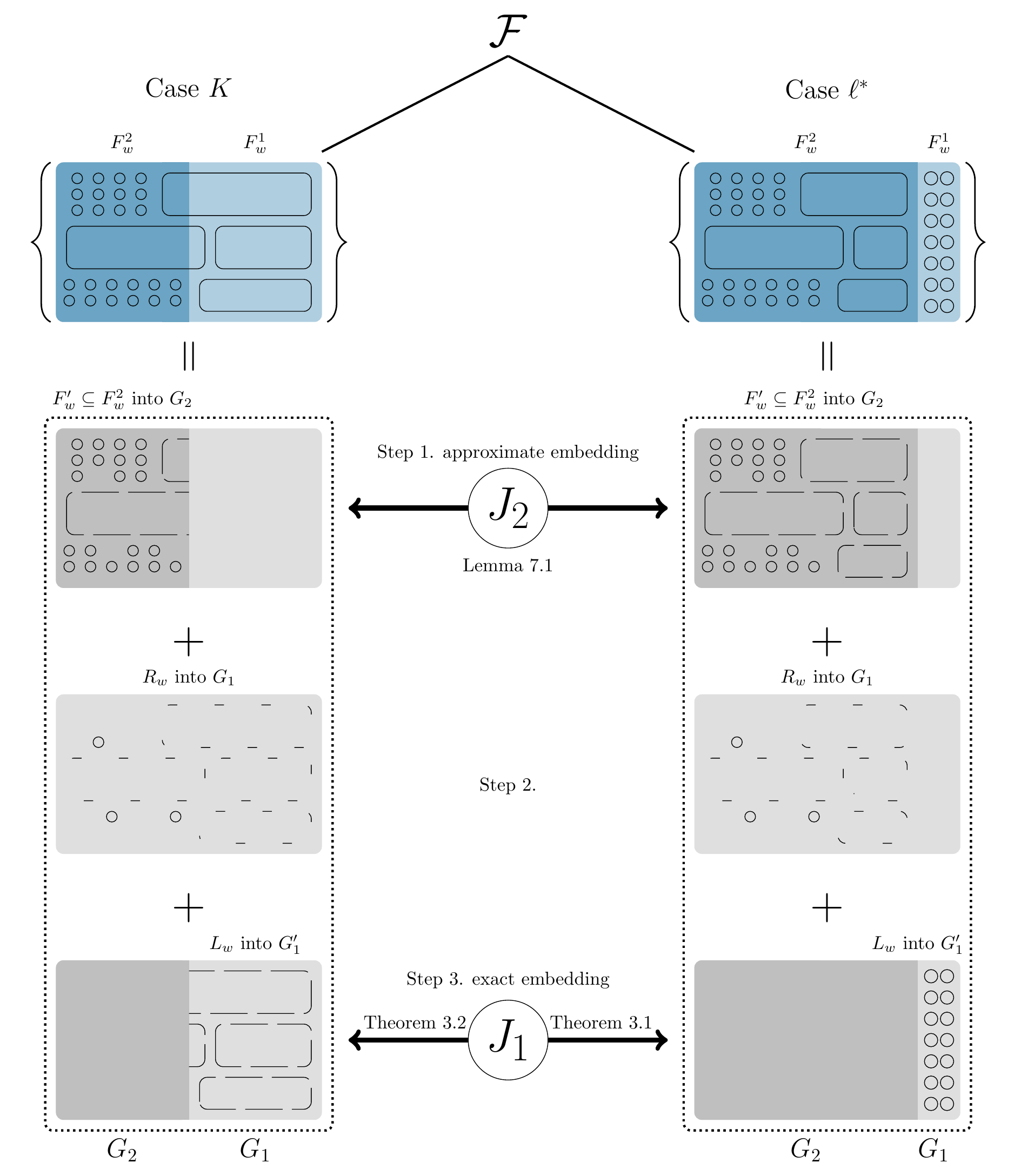}
\caption{An overview of the proof.}
\end{figure}

\section{Wheel decompositions} \label{sec:wheel}

In this section we describe the results we need on wheel
decompositions and how they follow from \cite{K2}.
We start by recalling the coloured wheels
described in section \ref{sec:over}.

For any $c \ge 3$, the uncoloured $c$-wheel consists 
of a directed $c$-cycle (called the rim),
another vertex (called the hub),
and an arc from each rim vertex to the hub.
We obtain the coloured $c$-wheel $\wC$
by giving all arcs colour $0$ except 
that one of the spokes has colour $c$.
We obtain the special $c$-wheel $\wKc$ 
by giving all arcs colour $0$ except 
that one rim edge $\ova{xy}$ 
and one spoke $\ova{yw}$ have colour $K$.
As discussed in section \ref{sec:over},
we will only use $\wKc$ with $c=8$, but here 
we will consider the general configuration
so that the decomposition problems are quite similar.
We start by stating the result for $\wC$.

\begin{center}
\includegraphics{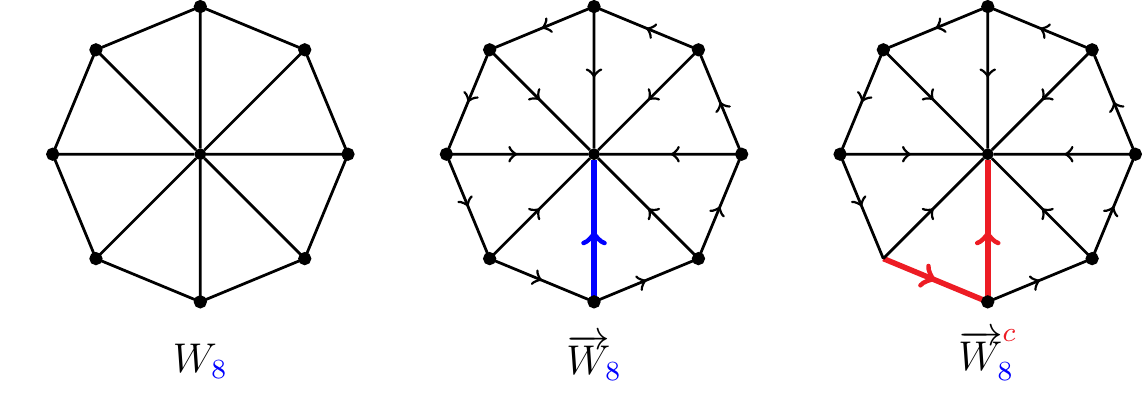}
\end{center}

\begin{theo} \label{decompL}
Let $n^{-1} \ll \dD \ll \oO \ll c^{-1}$ and $h=2^{50c^3}$.
Let $J = J^0 \cup J^c$ be a digraph with arcs coloured
$0$ or $c$, with $V(J)$ partitioned as $(V,W)$
where $\oO n \le |V|, |W| \le n$. 
Then $J$ has a $\wC$-decomposition such that
every hub lies in $W$ if the following hold:

{\em Divisibility:}
all arcs in $J[V]$ have colour $0$,
all arcs in $J[V,W]$ point towards $W$,
$d_J^-(v,V)=d_J^+(v,V)=d_J^+(v,W)$ for all $v \in V$,
and $d^-_J(w) = cd^-_{J^c}(w)$ for all $w \in W$.

{\em Regularity:}
each copy of $\wC$ in $J$ has a weight 
in $[\oO n^{1-c}, \oO^{-1} n^{1-c}]$ such that
for any arc $\ova{e}$ there is total weight $1 \pm \dD$ 
on wheels containing $\ova{e}$.

{\em Extendability:}
for all disjoint $A,B \sub V$ and $C \sub W$
each of size $\le h$ we have
$|N^+_{J^0}(A) \cap N^+_{J^c}(B) \cap W| \ge \oO n$
and $|N^+_{J^0}(A) \cap N^-_{J^0}(B) 
\cap N^-_{J^{c'}}(C)| \ge \oO n$
for both $c' \in \{0,c\}$.
\end{theo}

Before stating our result on $\wK$-decompositions,
we recall that $V$ has a cyclic order, which we can 
identify with the natural cyclic order on $[n]$,
and define the following separation properties.

\begin{defn} \label{def:sep}
For $1 \le x<y \le n$ the cyclic distance is
$d(x,y) = \min\{y-x,n+x-y\}$.
We say that $S \sub [n]$ is $d$-separated
if $d(a,a') \ge d$ for all distinct $a,a'$ in $S$.
For disjoint $S,S' \sub [n]$ we say $(S,S')$
is $d$-separated if $d(a,a') \ge d$ 
for all $a \in S$, $a' \in S'$.
\end{defn}

Now we state our result on $\wK$-decompositions.
We note that it only concerns digraphs $J$ such that 
$d(x,y) \ge d$ for all $\ova{xy} \in J[V]$,
as this is implied by the regularity assumption.
Our proof of Theorem \ref{main} will require us 
to only consider such $J$, so that we
can satisfy the extendability assumption.

\begin{theo} \label{decompK}
Let $n^{-1} \ll \dD \ll \oO \ll c^{-1}$. 
Let $h=2^{50c^3}$ and $d \ll n$.
Let $J = J^0 \cup J^K$ be a digraph with arcs coloured
$0$ or $K$, with $V(J)$ partitioned as $(V,W)$
where $\oO n \le |V|, |W| \le n$, such that
all arcs in $J[V,W]$ point towards $W$ and $J[W]=\es$.
Then $J$ has a $\wKc$-decomposition such that
every hub lies in $W$ if the following hold:

{\em Divisibility:}
$d^-_J(w) = cd^-_{J^K}(w)$ for all $w \in W$,
and for all $v \in V$ we have
$d_J^-(v,V)=d_J^+(v,V)=d_J^+(v,W)$
and $d^-_{J^K}(v,V)=d^+_{J^K}(v,W)$.

{\em Regularity:}
each $3d$-separated copy of $\wKc$ in $J$ has a weight 
in $[\oO n^{1-c}, \oO^{-1} n^{1-c}]$ such that
for any arc $\ova{e}$ there is total weight $1 \pm \dD$ 
on wheels containing $\ova{e}$.

{\em Extendability:}
for all disjoint $A,B \sub V$ and $L \sub W$
each of size $\le h$,
for any $a, b, \ell \in \{0,K\}$ we have
$|N^+_{J^a}(A) \cap N^-_{J^b}(B) 
\cap N^-_{J^\ell}(L)| \ge \oO n$,
and furthermore, if $(A,B)$ is $3d$-separated then
$|N^+_{J^0}(A) \cap N^+_{J^K}(B) \cap W| \ge \oO n$.
\end{theo}

For the remainder of this section we will explain
how Theorem \ref{decompK} follows from \cite{K2}
(we omit the similar and simpler details
for Theorem \ref{decompL}). We follow the exposition 
in \cite{K3}, which deduces from \cite{K2} a general 
result on coloured directed designs that we will apply here.

\subsection{The functional encoding}

We encode any digraph $J$ by a set of functions $\mf{J}$,
where for each arc $\ova{ab} \in J$ we include in $\mf{J}$
the function $(1 \mapsto a, 2 \mapsto b)$, i.e.\ the 
function $f:[2] \to V(J)$ with $f(1)=a$ and $f(2)=b$.
We will identify $\mf{J}$ with its characteristic vector,
i.e.\ $\mf{J}_f = 1_{f \in \mf{J}}$; if we want to emphasise
the vector interpretation we write $\ul{\mf{J}}$.
If $J$ has coloured arcs, and $\ell$ is a colour,
we write $J^\ell$ for the digraph in colour $\ell$,
which is encoded by $\mf{J}^\ell$.

We will consider decompositions by a coloured digraph $H$
defined as follows. We start with $\wKc$ on the vertex set
$[c+1]$, where we label the rim cycle by $[c]$ cyclically
(so $c+1$ is the hub) so that, writing $c_-=c-1$ and $c_+=c+1$, 
$\ova{c_- c}$ and $\ova{c c_+}$ have colour $K$ 
and all other arcs have colour $0$.
We let $\mc{P}$ be the partition $([c],\{c_+\})$ of $[c+1]$.
We introduce new colours $0'$ and $K'$, and change the colours
of $\ova{c c_+}$ to $K'$ and of the other spokes to $0'$.
We do this so that $H$ is `$(\mc{P},\text{id})$-canonical'
in the sense of \cite[Definition 7.1]{K3}; 
specialised to our setting, the relevant properties
are that $H$ is an oriented graph 
(with no multiple edges or $2$-cycles) 
and that for each colour all of its arcs have one fixed
pattern with respect to $\mc{P}$
(specifically, for colours $0$ and $K$
all arcs are contained in $[c]$, 
and for colours $0'$ and $K'$
all arcs are directed from $[c]$ to $\{c_+\}$).

Now we translate the $H$-decomposition problem 
for a digraph $J$ into its functional encoding.
We will have a partition $\mc{Q}=(V,W)$ of $V(J)$, 
and wish to decompose $J$ by copies $\phi(H)$ of $H$
such that $\phi([c]) \sub V$ and $\phi(c_+) \in W$
(i.e.\ wheels with hub in $W$),
and $\phi([c])$ is $3d$-separated
(in which case we will say that the graph $\phi(H)$ is $3d$-separated).
We think of the 
functional encoding $\mf{J}$ as living inside a
`labelled complex' $\Phi$ of all possible partial
embeddings of $H$: we define $\Phi = (\Phi_B: B \sub [c+1])$,
where each $\Phi_B$ consists of all injections 
$\phi:B \to V(J)$ such that $\phi(B \cap [c]) \sub V$,
$\phi(B \cap \{c_+\}) \sub W$ and $Im(\phi)$ is $3d$-separated.
The set of functional encodings of possible embeddings 
of $H$ (if present in $\mf{J}$) is then
\[ H(\Phi) := \{ \phi \mf{H} : \phi \in \Phi_{[c+1]} \},
\quad \text{where } 
\phi \mf{H} := \{ \phi \circ \tT: \tT \in \mf{H} \}.\]
The $H$-decomposition problem for $J$ is equivalent
to finding $\mc{H} \sub H(\Phi)$ with 
$\sum \{ \ul{\mf{H}'}: \mf{H}' \in \mc{H} \} = \ul{\mf{J}}$,
or equivalently $\bigcup \mc{H} = \mf{J}$
(where if $\mf{J}$ has multiple edges 
we consider a multiset union).
We call such $\mc{H}$ an $H$-decomposition in $\Phi$.

\subsection{Regularity}

Now we will describe the hypotheses of the theorem
that will give us an $H$-decomposition in $\Phi$.
We start with regularity, which is simply the 
functional encoding of the regularity assumption
in Theorem \ref{decompK}. Specifically, we say 
$J$ is $(H,\dD,\oO)$-regular in $\Phi$ if there are 
weights $y_\phi \in [\oO n^{1-c}, \oO^{-1} n^{1-c}]$
for each $\phi \in \Phi_{[c+1]}$ 
with $\phi \mf{H} \sub \mf{J}$ such that 
$\sum_\phi y_\phi \ul{\phi \mf{H}} = (1 \pm \dD)\ul{\mf{J}}$.

\subsection{Extendability}

Next we consider extendability, which we discuss 
in a simplified setting that suffices for our purposes,
following \cite[Definition 7.3]{K3}.
The idea is that for any vertex $x$ of $H$
there should be many ways to extend certain sets 
of partial embeddings of $H-x$ to embeddings of $H$.
Specifically, we say $(\Phi,J)$ is $(\oO,h,H)$-vertex-extendable
if for any $x \in [c+1]$ and disjoint $A_i \sub V \cup W$
for $i \in [c+1] \sm \{x\}$ each of size $\le h$ such that 
$(i \mapsto v_i: i \in [c+1] \sm \{x\}) \in \Phi$
whenever each $v_i \in A_i$, there are at least $\oO n$
vertices $v$ such that
\begin{enumerate}
\item $(i \mapsto v_i: i \in [c+1]) \in \Phi$
whenever $v_x=v$ and $v_i \in A_i$ for each $i \ne x$, and
\item each $\mf{J}^\ell$ with $\ell \in \{0,K,0',K'\}$ 
contains all $(1 \mapsto v_1, 2 \mapsto v_2)$
where for some $\tT \in \mf{H}^\ell$ we have
($v_1=v \ \& \ v_2 \in A_{\tT(2)}$) or
($v_2=v \ \& \ v_1 \in A_{\tT(1)}$).
\end{enumerate}
Note that by definition of $\Phi$ this only concerns 
maps $\phi$ such that $Im(\phi)$ is $3d$-separated.
To interpret (ii) we consider $4$ cases 
according to the position of $x$ in the wheel.
\begin{description}
\item[$x=c+1$.]
For any pairwise $3d$-separated $A_i \sub V$, $i \in [c]$ of sizes
$\le h$ there are at least $\oO n$ vertices $v$
such that $\ova{v_c v} \in J^{K'}$ for all $v_c \in A_c$
and $\ova{v_i v} \in J^{0'}$ for all $v_i \in A_i$, $i \ne c$.
Equivalently, for any disjoint $A,B \sub V$
with $|A| \le h$ and $|B| \le (c-1)h$
such that $(A,B)$ is $3d$-separated we have 
$|N^+_{J^{K'}}(A) \cap N^+_{J^{0'}}(B)| \ge \oO n$.
\item[$x=c$.]
For any pairwise $3d$-separated $A_i \sub V$, $i \in [c-1]$ 
and $A_{c+1} \sub W$ of sizes $\le h$ 
there are at least $\oO n$ vertices $v$ such that 
$\ova{v v_{c+1}} \in J^{K'}$ for all $v_{c+1} \in A_{c+1}$,
$\ova{v_{c-1} v} \in J^K$ for all $v_{c-1} \in A_{c-1}$,
and $\ova{vv_1} \in J^0$ for all $v_1 \in A_1$.
Equivalently, for any disjoint $A,B \sub V$ and $C \sub W$
of sizes $\le h$ such that $(A,B)$ is $3d$-separated we have
$|N^+_{J^K}(A) \cap N^-_{J^0}(B) \cap N^-_{J^{K'}}(C)| \ge \oO n$.
\item[$x=c-1$.]
For any pairwise $3d$-separated
$A_i \sub V$, $i \in [c] \sm \{c-1\}$
and $A_{c+1} \sub W$ of sizes $\le h$ 
there are at least $\oO n$ vertices $v$ such that 
$\ova{v v_{c+1}} \in J^{0'}$ for all $v_{c+1} \in A_{c+1}$,
$\ova{v v_c} \in J^K$ for all $v_c \in A_c$,
and $\ova{v_{c-2} v} \in J^0$ for all $v_{c-2} \in A_{c-2}$.
Equivalently, for any disjoint $A,B \sub V$ and $C \sub W$
of sizes $\le h$ such that $(A,B)$ is $3d$-separated we have
$|N^-_{J^K}(A) \cap N^+_{J^0}(B) \cap N^-_{J^{0'}}(C)| \ge \oO n$.
\item[$x \in \brak{c-2}$.]
For any pairwise $3d$-separated
$A_i \sub V$, $i \in [c] \sm \{x\}$
and $A_{c+1} \sub W$ of sizes $\le h$ 
there are at least $\oO n$ vertices $v$ such that 
$\ova{v v_{c+1}} \in J^{0'}$ for all $v_{c+1} \in A_{c+1}$,
$\ova{v v_{x+1}} \in J^0$ for all $v_{x+1} \in A_{x+1}$,
and $\ova{v_{x-1} v} \in J^0$ for all $v_{x-1} \in A_{x-1}$,
where $A_0 := A_c$.
Equivalently, for any disjoint $A,B \sub V$ and $C \sub W$
of sizes $\le h$ such that $(A,B)$ is $3d$-separated we have
$|N^-_{J^0}(A) \cap N^+_{J^0}(B) \cap N^-_{J^{0'}}(C)| \ge \oO n$.
\end{description}
All of these conditions follow from the extendability
assumption in Theorem \ref{decompK} 
(after renaming colours $0$ and $K$
in $J[V,W]$ as $0'$ and $K'$, and replacing $h$ with $(c-1)h$).

\subsection{Divisibility}

It remains to consider divisibility;
we follow \cite[Definition 7.2]{K3}.
For integers $s \le t$ we write $I^s_t$ 
for the set of injections from $[s]$ to $[t]$.
We identify $V \cup W$ with $[n']$ for some $n'$. 
For $0 \le i \le 2$, $\psi \in I^i_{n'}$, $\tT \in I^i_{c+1}$,
we define index vectors in $\mb{N}^2$ describing types 
with respect to the partitions $\mc{P}$ or $\mc{Q}$: we write 
$i_{\mc{P}}(\tT) = (|Im(\tT) \cap [c]|,|Im(\tT) \cap \{c_+\}|)$ 
and $i_{\mc{Q}}(\psi) = (|Im(\psi) \cap V|,|Im(\psi) \cap W|)$.
For example, for $\tT = (1 \mapsto c_-, 2 \mapsto c) \in \mf{H}$ 
we have $i_{\mc{P}}(\tT) = (2,0)$. We define degree vectors
$\mf{H}(\tT)^*$ and $\mf{J}(\psi)^*$ in $\mb{N}^{C \times I^i_2}$ by 
\[ \mf{H}(\tT)^*_{\ell\pi}=|\mf{H}^\ell(\tT\pi^{-1})|
\ \ \text{ and } \ \
\mf{J}(\psi)^*_{\ell\pi}=|\mf{J}^\ell(\psi\pi^{-1})|, \]
where e.g.\ $\mf{H}^\ell(\tT\pi^{-1})$ denotes the set of 
$\tT' \in \mf{H}^\ell$ having $\tT\pi^{-1}$ as a restriction.
Letting $\sgen{\cdot}$ denote the integer span of a set of vectors,
we say $J$ is $H$-divisible in $\Phi$ if
\[ \mf{J}(\psi)^* \in \sgen{\mf{H}(\tT)^*: 
i_{\mc{P}}(\tT) = i_{\mc{Q}}(\psi) } 
\ \ \text{ for all } \psi \in \Phi. \]
We refer to the divisibility conditions for index vectors
$(i_1,i_2)$ with $i_1+i_2=j$ as $j$-divisibility conditions,
where we assume $0 \le j \le 2$, as otherwise they are vacuous.
We describe these conditions concretely as follows.

{\bf $2$-divisibility.}
These conditions simply say that the arcs of $J$
have the same types with respect to $\mc{Q}$
as those of $H$ do with respect to $\mc{P}$,
i.e.\ all arcs of $J[V]$ have colour $0$ or $K$,
all arcs of $J[V,W]$ have colour $0'$ or $K'$, and $J[W]=\es$. 
To see this, consider any degree vector $\mf{H}(\tT)^*$
with $\tT \in I^2_{c+1}$. We write 
$\text{id} = (1 \mapsto 1, 2 \mapsto 2)$ and
$(12) = (1 \mapsto 2, 2 \mapsto 1)$.
For any $\ell \in C$, $\pi \in \{\text{id},(12)\}$
we have $\mf{H}(\tT)^*_{\ell \pi}$ equal to $1$
if $(\ell,\pi)$ is the pair such that
$\tT \circ \pi^{-1} \in \mf{H}^\ell$
(there is at most one such pair)
or equal to $0$ otherwise.
For example, if $\tT = (1 \mapsto c, 2 \mapsto c_-)$
then $\mf{H}(\tT)^*_{\ell \pi}$ is $1$
if $(\ell,\pi)=(K,(12))$, otherwise $0$.
Thus $\mf{H}\sgen{(i_1,i_2)} :=
\sgen{\mf{H}(\tT)^*: i_{\mc{P}}(\tT) = (i_1,i_2)}$
consists of all integer vectors supported in coordinates
with colours in $\{0,K\}$ if $(i_1,i_2)=(2,0)$
or $\{0',K'\}$ if $(i_1,i_2)=(1,1)$, whereas
$\mf{H}\sgen{(0,2)}$ only contains the all-$0$ vector.
Therefore, the $2$-divisibility conditions say that 
$\mf{J}(\psi)^*$ can be non-zero only at coordinates
with colours in $\{0,K\}$ if $i_{\mc{Q}}(\psi)=(2,0)$
or $\{0',K'\}$ if $i_{\mc{Q}}(\psi)=(1,1)$, 
and $\mf{J}(\psi)^*=0$ if $i_{\mc{Q}}(\psi)=(0,2)$,
i.e.\ $J$ has the same arc types with respect to $\mc{Q}$
as $H$ with respect to $\mc{P}$.

{\bf $0$-divisibility.}
Writing $\es$ for the function with empty domain,
all $\mf{H}(\es)^*_{\ell \es} = |\mf{H}^\ell|=|H^\ell|$,
and similarly for $J$, so the $0$-divisibility condition 
is that for some integer $m$ all $|J^\ell|=m|H^\ell|$.
For our specific $H$, this is equivalent to
$|J[V]|=|J[V,W]|=c|J^c[V]|=c|J^c[V,W]|$.

{\bf $1$-divisibility.}
Given $\tT=(1 \mapsto a) \in I^1_{c+1}$ 
and $\ell \in C =\{0,K,0',K'\}$,
the two coordinates of $\mf{H}(\tT)^*$
corresponding to colour $\ell$ 
are the in/outdegrees of $a$ in $H^\ell$: 
we have $\mf{H}(\tT)^*_{\ell \text{id}} 
= |\mf{H}(1 \mapsto a)| = d^+_{H^\ell}(a)$
and $\mf{H}(\tT)^*_{\ell (12)} 
= |\mf{H}(2 \mapsto a)| = d^-_{H^\ell}(a)$.
Similarly, for $\psi=(1 \mapsto v) \in I^1_{n'}$
the coordinates of $\mf{J}(\psi)^*$ corresponding 
to colour $\ell$ are $d^\pm_{J^\ell}(v)$. We compute:

\begin{tabular}{LCCCCCCCC}
\mf{H}(1 \mapsto a)^*  & d^+_{H^0}(a)  & d^-_{H^0}(a)
 & d^+_{H^K}(a)  & d^-_{H^K}(a) & d^+_{H^{0'}}(a)
 & d^-_{H^{0'}}(a) & d^+_{H^{K'}}(a) & d^-_{H^{K'}}(a)  \\
a = c_+  & 0 & 0 & 0 & 0 & 0 & c-1 & 0 & 1 \\ 
a = c  & 1 & 0 & 0 & 1 & 0 & 0 & 1 & 0 \\
a = c_-  & 0 & 1 & 1 & 0 & 1 & 0 & 0 & 0 \\
a \in [c-2]  & 1 & 1 & 0 & 0 & 1 & 0 & 0 & 0
\end{tabular}
\begin{align*}
\text{ so } \ \ \sgen{\mf{H}(1 \mapsto c_+)^*}
& = \{ \bm{v} \in \mathbb{Z}^8 : v_1=v_2=v_3=v_4=v_5=v_7=0, 
v_6=(c-1)v_8 \}, \text{ and }  \\
 \sgen{\mf{H}(1 \mapsto a)^*: a \in [c]}
& = \{ \bm{v} \in \mathbb{Z}^8 : v_2=v_5, v_4=v_7, 
v_1+v_3=v_2+v_4, v_6=v_8=0 \}.
\end{align*}
For $w \in W$ the $1$-divisibility condition
is $\mf{J}(1 \mapsto w)^* \in \sgen{\mf{H}(1 \mapsto c_+)^*}$,
i.e.\ $d^-_{J^{0'}}(w) = (c-1)d^-_{J^{K'}}(w)$,
or equivalently $d^-_J(w) = cd^-_{J^{K'}}(w)$.
For $v \in V$ the $1$-divisibility condition
is $\mf{J}(1 \mapsto v)^* \in 
\sgen{\mf{H}(1 \mapsto a)^*: a \in [c]}$,
which is equivalent to
$d^-_{J^K}(v) = d^+_{J^{K'}}(v)$ 
and $d^+_J(v,V) = d^-_J(v,V) = d^+_J(v,W)$.

All of these divisibility conditions follow from the 
divisibility assumption in Theorem \ref{decompK}
(after renaming colours $0$ and $K$
in $J[V,W]$ as $0'$ and $K'$).
By the above discussion, Theorem \ref{decompK}
follows from the following special case
of \cite[Theorem 7.4]{K3}.

\begin{theo}
Let $n^{-1} \ll \dD \ll \oO \ll c^{-1}$. 
Let $h=2^{50c^3}$ and $d \ll n$.
Let $J$ be a digraph with $V(J)$ partitioned as $(V,W)$
where $\oO n \le |V|, |W| \le n$, such that $J[W]=\es$,
all arcs in $J[V,W]$ point towards $W$,
all arcs in $J[V]$ are coloured $0$ or $K$
and all arcs in $J[V,W]$ are coloured $0'$ or $K'$.
Let $\Phi = (\Phi_B: B \sub [c+1])$,
where $\Phi_B$ consists of all injections 
$\phi:B \to V(J)$ such that $\phi(B \cap [c]) \sub V$,
$\phi(B \cap \{c_+\}) \sub W$ and $Im(\phi)$ is $3d$-separated.
Suppose $J$ is $H$-divisible in $\Phi$
and $(H,\dD,\oO)$-regular in $\Phi$ and
$(\Phi,J)$ is $(\oO,h,H)$-vertex-extendable.
Then $J$ has an $H$-decomposition in $\Phi$.
\end{theo}

\section{The algorithm} \label{sec:alg}

Suppose we are in the setting of Theorem \ref{main}:
we are given a $(\eps,t)$-typical $\aA n$-regular digraph $G$ 
on $n$ vertices, where $n^{-1} \ll \eps \ll t^{-1} \ll \aA$,
and we need to decompose $G$ into some given
family $\mc{F}$ of $\aA n$ oriented one-factors on $n$ vertices.
In this section we present an algorithm that partitions 
almost all of $G$ into two digraphs $G_1$ and $G_2$, 
and each factor $F_w$ into subfactors $F^1_w$ and $F^2_w$,
and also sets up auxiliary digraphs $J_1$ and $J_2$, such that 
(i) an approximate wheel decomposition of $J_2$
gives an approximate decomposition of $G_2$
into partial factors that are roughly $\{F^2_w\}$,
(ii) given the approximate decomposition of $G_2$,
we can set up (via a small additional greedy embedding)
the remaining problem to be finding an exact decomposition 
of a small perturbation $G'_1$ of $G_1$ into partial factors
that are roughly $\{F^1_w\}$, corresponding to a wheel
decomposition of a small perturbation $J'_1$ of $J_1$.
For most of the section we will describe and motivate 
the algorithm; we then conclude with the formal statement.

We fix additional parameters with hierarchy
\begin{equation} \label{hierarchy}
n^{-1} \ll \eps \ll t^{-1} \ll K^{-1} \ll d^{-1}
\ll \eta \ll s^{-1} \ll L^{-1} \ll \aA.
\end{equation}
For convenient reference later, we also make some comments 
here regarding the roles of these additional parameters:
$\eta$ will be used to bound the number of vertices embedded greedily,
we consider a cycle `long' if it has length at least $K$,
and the cyclic intervals used to define the special colour $K$
will have sizes $d_i = d/(2s)^{i-1}$ with $i \in [2s+1]$.
By the reductions in section \ref{sec:red}, we will be able 
to assume that we are in one of the following cases:
 
Case $K$: each $F \in \mc{F}$ has at least $n/2$ vertices 
in cycles of length at least $K$,  

Case $\ell^*$ with $\ell^* \in [3,L]$: each $F \in \mc{F}$ 
has $\ge L^{-3} n$ cycles of length $\ell^*$.

We write $\mc{F}=(F_w: w \in W)$, so $|W| = \aA n$.
We partition each $F_w$ as $F^1_w \cup F^2_w$ as follows.
In Case $\ell^*$ we let $F^1_w$ consist of exactly
$L^{-3} n$ cycles of length $\ell^*$ 
(and then $F^2_w = F_w \sm F^1_w$).
In Case $K$ we choose $F^1_w$ 
with $|F^1_w| - n/2 \in [0,2K]$
to consist of some cycles of length at least $K$
and at most one path of length at least $K$.
To see that this is possible,
consider any induced subgraph $F'_w$ of $F_w$ 
with $|F'_w|=n/2+K$ obtained by greedily adding cycles 
of length at least $K$ until the size is at least $n/2 + K$, 
and then deleting a (possibly empty) path from one cycle.
Let $P_1$ and $P_2$ denote the two paths 
of the (possibly) split cycle, where $P_1 \in F'_w$.
If $|P_1|, |P_2| \ge K$ we let $F^1_w=F'_w$.
If $|P_1| < K$ we let $F^1_w=F'_w \sm P_1$.
If $|P_2| < K$ we let $F^1_w=F'_w \cup P_2$.
In all cases, $F^1_w$ is as required.

The algorithm is randomised, so we start by defining 
probability parameters. The graphs $G_1$ and $G_2$
are binomial random subdigraphs of $G$ of sizes 
that are slightly less than one would expect
(we leave space for a greedy embedding that will occur
between the approximate decomposition step
and the exact decomposition step).
For each $w \in W$ we let
$p^g_w = (1-\eta) n^{-1}|F^g_w| + n^{-.2}$
(so $1-\eta \le p^1_w+p^2_w \le 1-L^{-3}\eta$). 
We let $p_g = |W|^{-1}\sum_{w \in W} p^g_w$ 
(so $1-\eta \le p_1+p_2 \le 1-L^{-3}\eta$).
For each arc $e$ of $G$ independently we will let 
$\mb{P}(e \in G_g) = p_g$ for $g \in [2]$.

We introduce further probabilities corresponding
to the cycle distributions of each $F^g_w$.
For $c<K$ we write $q^g_{w,c} n$ for the number 
of cycles of length $c$ in $F^g_w$
and let $p^g_{w,c} = (1-\eta) q^g_{w,c}$.
We define $p^g_{w,K}$ so that $F^g_w$ 
has about $8p^g_{w,K}n$ vertices
not contained in cycles of length $<K$
(for technical reasons, we also ensure that
each $p^g_{w,K} \ge n^{-.2}$, which explains
the term $n^{-.2}$ in the definition of $p^g_w$).
Averaging over $W$ gives the corresponding probabilities
that describe the uses of arcs in each $G_g$:
we let $p^g_c = |W|^{-1} \sum_{w \in W} p^g_{w,c}$
so that for each $c<K$, the number of edges in $G_g$
allocated to cycles of length $c$ will be roughly
$\sum_{w \in W} cp^g_{w,c} n = |W| cp^g_c n
= \aA cp^g_c n^2 = cp^g_c |G| + O(n)$,
and similarly, roughly $8p^g_K |G| + O(n)$ arcs
in $G_g$ will be allocated to long cycles.

The remainder of the algorithm is concerned
with the auxiliary digraphs $J_g$.
For any colour $c$, we let $J^c_g$ 
denote the arcs of colour $c$ in $J_g$.
We also write $J^*_g = \cup_{c \ne K} J^c_g$.
First we consider arcs within $J_g[V]$. 
Throughout the paper, we fix a  cyclic order on $V$, 
which we choose uniformly at random. 
For $v \in V$, let $v^+$ denote the successor of $v$
and $v^-$ denote the predecessor of $V$.
Arcs of the special colour $K$ should correspond 
to $1/8$ of the factor arcs that are not in short cycles,
so should form a graph of density about $p^g_K$.
For each arc $\ova{xy} \in G_g$ not of the form
$\ova{zz}^+$ (to avoid loops, we don't mind double edges) 
independently we assign $\ova{xy}$ to 
colour $K$ with probability $p^g_K/p_g$
or colour $0$ with probability $p^g_*/p_g$
(where $p^g_K + p^g_*$ is slightly less than $p_g$).
If $\ova{xy}$ has colour $K$ we add
$\ova{xy}^-$ to $J^K_g$.

Now we consider $J_g[V,W]$.
These arcs are all directed from $V$ to $W$.
For each $w \in W$ and cycle length $c<K$, 
there should be about $cp^g_{w,c} n$
vertices available for the $c$-cycles in $F^g_w$.
The colouring of $\wC$ requires $1/c$-fraction 
of these to be joined to $w$ in colour $c$, so we 
should have $N^-_{J^c_g}(w) \approx p^g_{w,c} n$.
Similarly, there should be about $8p^g_{w,K} n$
vertices available for vertices of $F^g_w$ not in
short cycles, and the colouring of $\wK$ requires 
$1/8$ of these to be joined to $w$ in colour $c$, 
so we should have $N^-_{J^K_g}(w) \approx p^g_{w,K} n$.
These arcs are chosen randomly, not independently,
but according to a random collection of intervals,
of sizes $d_i = d/(2s)^{i-1}$ with $i \in [2s+1]$,
where $d$ is small enough that the resulting graph 
is roughly typical, but large enough to give a good
upper bound on the number of vertices in long cycles
that become unused when they are chopped up into paths,
and so need to be embedded greedily.

These intervals must be chosen quite carefully,
because of the following somewhat subtle constraint.
Recall that in Case $K$ we will reduce to a path factor
problem in some subdigraph $H$ of $G$. This can only have
a solution if each vertex $x$ has degree 
$d_H^\pm(x) = d_2(x) - d_{\pm}(x)$, where $d_2(x)$ 
is the number of path factors that will use $x$
and $d_-(x)$ (respectively $d_+(x)$)
is the number of these 
in which $x$ is the start (respectively end).
The path factors will be obtained from 
a set of arc-disjoint $\wK$'s, where for each $w \in W$,
its colour $K$ neighbourhood is given by a set of intervals
$([x^w_i,(y^w_i)^-]: i \in I_w)$, so its $\wK$'s will define 
paths from $x^w_i$ to $y^w_i$. Thus in the auxiliary
digraph $J$, the degree of $x$ into $W$ must be
$d^+_J(x,W) = d_2(x) - d'_1(x)$, where $d'_1(x)$
is the number of path factors in which
$x$ is some successor $(y^w_i)^+$.
To relate these two formulae,
we note that a wheel decomposition of $J$ requires 
$d^+_J(x,W)=d^+_J(x,V)=d^-_J(x,V)$
and $d^+_{J^K}(x,W)=d^-_{J^K}(x,V)$,
and that in the twisting construction,
$d^-_{J^K}(x^-,V)$ arcs of $H$ at $x$
are not counted by $d^-_J(x,V)$,
whereas $d^-_{J^K}(x,V)$ arcs of $H$ 
not at $x$ are counted by $d^-_J(x,V)$.
Writing $\DD(x) = d^-_{J^K}(x^-,V) - d^-_{J^K}(x,V)
= d^+_{J^K}(x^-,W) - d^+_{J^K}(x,W)$,
we deduce $d_H^+(x)=d^+_J(x,V)$ and
$d_H^-(x) = d^-_J(x,V) + \DD(x)$, 
so we need $\DD(x) = d'_1(x) - d_+(x)$
and $d_1'(x)=d_-(x)$.
So $\DD(x)=d_-(x)-d_+(x)$.
We will ensure that both sides are always $0$
(taking $H$ equal to the digraph $G'_1$ in which we need
to solve the path factor problem), i.e.\ 
\begin{enumerate}
\item every vertex is used equally often
as a startpoint or as a successor of an interval, and 
\item all vertices appear in some interval 
for the same number of factors.
\end{enumerate}

To achieve this, we identify $V$ with $[n]$ under the 
natural cyclic order, and select our intervals from
canonical sets $\mc{I}^i_j$, $i \in [2s+1]$, $j \in [d_i]$,
where each $\mc{I}^i_j$ is a partition of $[n]$
into $n/d_i \pm 1$ intervals of length at most $d_i$,
we have $\mc{I}^i_j \cap \mc{I}^i_{j'} = \es$ for $j \ne j'$,
and for each $i$, every $v \in [n]$ occurs exactly once as
a startpoint of some interval in $\mc{I}^i = \cup_j \mc{I}^i_j$,
and also exactly once as a successor of some interval in $\mc{I}^i$.
The two conditions discussed in the previous paragraph will then
be satisfied if there are numbers $t_i$, $i \in [2s+1]$ such that
every interval in $\mc{I}^i$ is used by exactly $t_i$ factors.
Each $w$ will select intervals from some $\mc{I}^{i(w)}_{j(w)}$, 
and these intervals must be non-consecutive, 
so that the paths do not join up into longer paths.
This explains why we use several different interval sizes: 
if we only used one size $d$ then a pair of vertices in $V$ 
at cyclic distance $d$ could never be both used for the same factor, 
and so we would be unable to satisfy the conditions of the
wheel decomposition results in section \ref{sec:wheel}.

Now we describe how factors choose intervals.
For each $w \in W$, we start by independently choosing 
$i=i(w) \in [2s+1]$ and $j=j(w) \in [d_i]$ uniformly at random.
Given $i$ and $j$, we activate each interval in $\mc{I}^i_j$
independently with probability $1/2$, and select any interval $I$ 
such that $I$ is activated, and its two neighbouring 
intervals $I^\pm$ are not activated.
We thus obtain a random set of non-consecutive intervals where 
each interval appears with probability $1/8$ (not independently).
We form random sets of intervals $\mc{X}^g_w$ where each
interval selected for $w$ is included in $\mc{X}^g_w$ 
independently with probability $8p^g_{w,K}$ (and is included 
in at most one of $\mc{X}^1_w$ or $\mc{X}^2_w$).
Thus, given $w \in W_i := \{w': i(w')=i\}$, any interval 
$I \in \mc{I}^i$ appears in $\mc{X}^g_w$ with probability
$p^g_{w,K}/d_i$. The events $\{I \in \mc{X}^g_w\}$ for $w \in W_i$ 
are independent, so whp about $\sum_{w \in W_i} p^g_{w,K}/d_i$
factors use $I$.

Our final sets of intervals $\mc{Y}^g_w$ are obtained from
$\mc{X}^g_w$ by removing a small number of intervals so that
every interval in $\mc{I}^i$ is used exactly $t^g_i$ times,
where $t^g_i$ is about $\sum_{w \in W_i} p^g_{w,K}/d_i$.
(We only need this property when $g=1$, but for uniformity
of the presentation we do the same thing for $g=2$.)
These intervals determine $J^K_g[V,W]$: we let
$N^-_{J^K_g}(w) = Y^g_w := \bigcup \mc{Y}^g_w$,
i.e.~the subset of $V$ which is 
the union of the intervals in $\mc{Y}^g_w$.
As each $x$ is the startpoint of exactly one interval
in $\mc{I}^i$ it occurs as the startpoint of an interval
for exactly $t_g := \sum_i t^g_i$ factors; the same
statement holds for successors of intervals.
As each $x \in V$ appears in exactly one interval
in each $\mc{I}^i_j$ we deduce $d^+_{J^K_g}(x,W)
= \sum_{i=1}^{2s+1} \sum_{j=1}^{d_i} t^g_i 
\approx \sum_{w \in W} p^g_{w,K} = |W| p^g_K$.

The other arcs of $J$ incident to $w$ will come from
$\ov{Y}_w := V \sm \big( Y^1_w \cup Y^2_w 
\cup (Y^1_w)^+ \cup (Y^2_w)^+ \big)$, 
where $(Y^g_w)^+$ is the set of successors of intervals 
in $\mc{Y}^g_w$ (these vertices are endpoints of paths so
should be avoided by the short cycles, and also by the $7/8$ 
of the paths not specified by the intervals).
We define $\ov{J}[V,W]$ by $N^-_{\ov{J}}(w)=\ov{Y}_w$.
For any $x \in V$ we will have
$\mb{P}(x \in Y_w^g) \approx
\mb{P}(x \in X_w^g) = p_{w,K}^g$
and $\mb{P}(x \in Y_w^g \mid w \in W_i) \approx
\mb{P}(x \in X_w^g \mid w \in W_i) = p_{w,K}^g/d_i$,
so $|\ov{Y}_w| \approx \ov{p}_w n$,
where $\ov{p}_w = 1 - \tfrac{d_i+1}{d_i} (p_{w,K}^1+p_{w,K}^2)$.

In $J^*_g = J_g \sm J^K_g$ we require about $p^g_{w,*} n$ 
such arcs, where $p^g_{w,*} := p^g_w - p^g_{w,K}$,
and of these, for each cycle length $c<K$
we require about $p^g_{w,c} n$ arcs of colour $c$.
For each $x \in \ov{Y}_w$ independently we include
the arc $xw$ in at most one of the $J^*_g$ 
with probability $p^g_{w,*}/\ov{p}_w$,
which is a valid probability as
$p^1_{w,*} + p^2_{w,*} 
= 1 - L^{-3}\eta - p_{w,K}^1 - p_{w,K}^2 < \ov{p}_w$.
Then we give each $xw \in J^*_g[V,W]$
colour $c$ with probability $p^g_{w,c}/p^g_{w,*}$. 
In particular, $xw$ in $J^*_g$ is coloured $0$ 
with probability $p^g_{w,0}/p^g_{w,*}$, where
$p^g_{w,0} := p^g_{w,*} - \sum_{c=3}^{K-1} p^g_{w,c}$.

\subsection{Formal statement of the algorithm}

The input to the algorithm consists of
an $\aA n$-regular digraph $G$ on $V$,
a family $(F_w: w \in W)$ of $\aA n$ oriented one-factors, 
each partitioned as $F_w = F^1_w \cup F^2_w$,
and parameters satisfying
$n^{-1} \ll \eps \ll t^{-1} \ll K^{-1} 
\ll d^{-1} \ll \eta \ll s^{-1} \ll L^{-1} \ll \aA$.
We identify $V$ with $[n]$ according 
to a uniformly random bijection and
adopt the natural cyclic order on $[n]$:
each $x \in [n]$ has successor $x^+=x+1$ (where $n+1$ means 
$1$) and predecessor $x^-=x-1$ (where $0$ means $n$).
Let $d_i=d/(2s)^{i-1}$ for $i \in [2s+1]$.
We write $n=r_id_i+s_i$ with $r_i \in \mb{N}$
and $0 \le s_i < d_i$, and let
\[ P^i_j = \left\{ \begin{array}{ll}
\{ kd_i+j: 0 \le k \le r_i \} 
& \text{if } j \in [s_i], \\
\{ kd_i+j: 0 \le k \le r_i-1 \} 
& \text{if } j \in [d_i] \sm [s_i].
\end{array} \right. \] 
For each $i \in [s+1]$ and $j \in [d_i]$
we define a partition of $[n]$ into a family
of cyclic intervals $\mc{I}^i_j$ defined 
as all $[a,b^-]$ where $a \in P^i_j$ and $b$ is
the next element of $P^i_j$ in the cyclic order.
(So $|\mc{I}^i_j|=n/d_i \pm 1$,
each $I \in \mc{I}^i_j$ has $|I| \le d_i$, and 
$\mc{I}^i_j \cap \mc{I}^i_{j'} = \es$ for $j \ne j'$.)
We let $\mc{I}^i = \cup_{j \in [d_i]} \mc{I}^i_j$.
(So for every $v \in [n]$, 
exactly one $[a,b^-] \in \mc{I}^i$ has $a=v$,
and exactly one $[a,b^-] \in \mc{I}^i$ has $b=v$.)
Each $w \in W$ will be assigned $i(w) \in [2s+1]$.
For $c<K$ write $q^g_{w,c} n$ for the number 
of cycles of length $c$ in $F^g_w$. Let
\begin{gather*}
p^g_w = (1-\eta) n^{-1}|F^g_w| + n^{-.2}, \qquad
p^g_{w,c} = (1-\eta) q^g_{w,c}
\ \text{ for } 3 \le c < K, \qquad
p^g_{w,K}  = \tfrac{1}{8} \left( p^g_w 
- \Ss_{c=3}^{K-1} cp^g_{w,c} \right), \\
p^g_{w,*} = p^g_w - p^g_{w,K}, \qquad
p^g_{w,0} = p^g_{w,*} - \Ss_{c=3}^{K-1} p^g_{w,c},  \qquad
p_{w,K} = p^1_{w,K} + p^2_{w,K}, \\
\ov{p}_w = 1 - \tfrac{d_{i(w)}+1}{d_{i(w)}} p_{w,K}, \qquad
p_g  = |W|^{-1} \Ss_{w \in W} p^g_w, \qquad
p^g_c  = |W|^{-1} \Ss_{w \in W} p^g_{w,c}
\ \text{ for } c \in [0,K] \cup \{*\}.
\end{gather*}

We complete the algorithm by applying the 
following subroutines INTERVALS and DIGRAPH.

\vspace{-0.3cm} \begin{center} 
INTERVALS \end{center} \vspace{-0.3cm}

\begin{enumerate}
\item For each $w \in W$ independently choose
$i(w) \in [2s+1]$ and $j(w) \in [d_{i(w)}]$ 
uniformly at random. Let $W_i = \{ w: i(w)=i \}$.
\item For each $w \in W$, let $\mc{A}_w$ include each interval
of $\mc{I}^{i(w)}_{j(w)}$ independently with probability $1/2$. \\
Let $\mc{S}_w$ consist of all $I \in \mc{A}_w$ such that
both neighbouring intervals $I^\pm$ of $I$ are not in $\mc{A}_w$.
\item Let $\mc{X}^g_w$, $g \in [2]$ be disjoint
with $\mb{P}(I \in \mc{X}^g_w)=8p^g_{w,K}$ 
independently for each $I \in \mc{S}_w$.
\item Let $t^g_i = \min \{ |\mc{X}^g(I)|: I \in \mc{I}^i \}$,
where $\mc{X}^g(I) := \{w \in W_i: I \in \mc{X}^g_w\}$,
and obtain $\mc{Y}^g_w \sub \mc{X}^g_w$ by deleting 
each $I \in \mc{I}^i$, $i \in [2s+1]$ from
$|\mc{X}^g(I)|-t^g_i$ sets $\mc{X}^g_w$ with $w \in \mc{X}^g(I)$, 
independently uniformly at random.
Write $\mc{Y}^g(I) := \{w \in W_i: I \in \mc{Y}^g_w\}$
(so $|\mc{Y}^g(I)|=t^g_i$ for $I \in \mc{I}^i$). 
\end{enumerate}

\vspace{-0.3cm} \begin{center} 
DIGRAPH \end{center} \vspace{-0.3cm}

\begin{enumerate}
\item Let $G_1$ and $G_2$ be arc-disjoint  with 
$\mb{P}(\ova{e} \in G_g) = p_g$ independently for each arc $\ova{e}$ of $G$.
\item For each $g \in [2]$ and $\ova{xy} \in G_g$ independently,
if $\ova{xy}$ is $\ova{zz}^-$ or $\ova{zz}^+$ for some $z$ 
add $\ova{xy}$ to $J^0_g$, otherwise choose exactly one of
$\mb{P}(\ova{xy} \in J^0_g) = p^g_*/p_g$ or
$\mb{P}(\ova{xy}^- \in J^K_g) = p^g_K/p_g$.
\item For each $w \in W$,
add $\ova{xw}$ to $J^K_g$ 
for each $x \in Y^g_w := \bigcup \mc{Y}^g_w$,
and add $\ova{xw}$ to $\ov{J}$
for each $x \in \ov{Y}_w := V \sm ( Y^1_w \cup Y^2_w 
\cup (Y^1_w)^+ \cup (Y^2_w)^+ )$.
\item For each arc $\ova{xw}$ of $\ov{J}[V,W]$ independently,
add $\ova{xw}$ to $J^*_g[V,W]$ with probability $p^g_{w,*}/\ov{p}_w$,
and give it exactly one colour $c \ne K$ (including $0$) 
with probability $p^g_{w,c}/p^g_{w,*}$.
\end{enumerate}

We conclude this section by recording some estimates 
on the algorithm parameters used throughout the paper.
\begin{gather*}
\text{In Case } K, \text{ all } 
|F^g_w| = n/2 \pm 2K, \quad p^1_w, p^2_w  > .49,
  \quad p^1_{w,K} = p^1_w/8 > 1/17, \\ 
 p^1_{w,c}=0 \text{ for } c \in [3,K-1], 
 \quad  p^1_{w,*} = p^1_{w,0} = 7p^1_w/8 > 1/3 
\quad  \text{ and } \
p^2_{w,*} \ge p^2_{w,0} \ge 2p^2_w/3 > 1/4. \\
\text{In Case  } \ell^*, \text{ all }
|F^1_w| = \ell^* L^{-3} n, 
\quad |F^2_w| = n - \ell^* L^{-3} n, 
\quad p^1_w  > (1-\eta)\ell^*L^{-3} > 2L^{-3}, \\ 
p^2_w > 1 - 2L^{-2} > .9, 
\quad p^1_{w,\ell^*} = p^1_w/\ell^* > .9L^{-3}, 
\quad p^1_{w,K}=n^{-.2}/8, 
\quad p^1_{w,c}=0 \text{ for } c \in [3,K-1] \sm \{\ell^*\}, \\
p^1_{w,*}>2L^{-3}, 
\quad p^1_{w,0} \ge 2p^1_{w,*}/3 > L^{-3} 
\quad \text{ and } \
p^2_{w,*} \ge p^2_{w,0} \ge 2p^2_w/3 > .6.\\
\text{In both cases,  } p^2_{w,K} \geq n^{-.2}/8.
\end{gather*}

\section{Analysis I: intervals} \label{sec:int}

In this section we analyse the families of intervals chosen 
by the INTERVALS subroutine in section \ref{sec:alg};
our goal is to establish various regularity and extendability
properties of $J^K_g[V,W]$ and $\ov{J}_g[V,W]$
(which are defined in step (iii) of DIGRAPH
but are completely determined by INTERVALS).
We also deduce some corresponding properties that follow 
from these under the random choices in DIGRAPH. 
Before starting the analysis, we state some standard results
on concentration of probability that will be used throughout
the remainder of the paper. We use the following classical 
inequality of Bernstein (see e.g.\ \cite[(2.10)]{BLM})
on sums of bounded independent random variables.
(In the special case of a sum of independent indicator
variables we will simply refer to the `Chernoff bound'.)

\begin{lemma} \label{bernstein}
Let $X = \sum_{i=1}^n X_i$ be a sum of
independent random variables with each $|X_i|<b$.

Let $v = \sum_{i=1}^n \mb{E}(X_i^2)$.
Then $\mb{P}(|X-\mb{E}X|>t) 
< 2e^{-t^2/2(v+bt/3)}$.
\end{lemma}

We also use McDiarmid's bounded differences inequality,
which follows from Azuma's martingale inequality
(see \cite[Theorem 6.2]{BLM}). 

\begin{defn} \label{def:vary}
Suppose $f:S \to \mb{R}$ where $S = \prod_{i=1}^n S_i$
and $b = (b_1,\dots,b_n) \in \mb{R}^n$.
We say that $f$ is \emph{$b$-Lipschitz} if for any 
$s,s' \in S$ that differ only in the $i$th coordinate
we have $|f(s)-f(s')| \le b_i$. 
We also say that $f$ is \emph{$v$-varying} 
where $v=\sum_{i=1}^n b_i^2/4$.
\end{defn}

\begin{lemma} \label{azuma}
Suppose $Z = (Z_1,\dots,Z_n)$ is a sequence 
of independent random variables,
and $X=f(Z)$, where $f$ is $v$-varying.
Then $\mb{P}(|X-\mb{E}X|>t) \le 2e^{-t^2/2v}$.
\end{lemma}

The next lemma records various regularity and extendability
properties of $J^K_g[V,W]$ and $\ov{J}_g[V,W]$.
We recall that each $N^-_{J^K_g}(w) = Y^g_w$ 
and $N^-_{\ov{J}_g}(w) = \ov{Y}_w$, and also
our notation for common neighbourhoods, 
e.g.\ $N^-_{J^K_g}(R) = \bigcap_{w \in R} N^-_{J^K_g}(w)$
in statement (iv). Statements (iv) and (v)
will be applied to $n^{O(1)}$ choices 
of set $U$ or function $h$, so their conclusions apply 
whp simultaneously to all these choices
(recalling our convention that `whp' refers to events
with exponentially small failure probability). 
For $x \in V$ we write
$t^-_g(x)$ or $t^+_g(x)$ for the number of $w$ 
such that $x$ is the startpoint or successor
of an interval in $\mc{Y}^g_w$.
We also use the separation property
from Definition \ref{def:sep}.

\begin{lemma} \label{lem:int}
Let $g \in [2]$, $U \sub V$ and $h:W \to \mb{R}$ 
with each $|h(w)|<n^{.01}$. Then whp:
\begin{enumerate}
\item 
$|\mc{Y}^g(I)| = t^g_i 
= \tfrac{|W| p^g_K}{(2s+1)d_i} \pm n^{.51}$ 
for all $I \in \mc{I}^i$, $i \in [2s+1]$.
\item $d^+_{J^K_g}(x,W) = |W| p^g_K \pm n^{.52}$
and $t^\pm_g(x) = t_g := \sum_i t^g_i$
for each $x \in V$. 
\item $d^-_{J^K_g}(w) = |Y^g_w| = p^g_{w,K} n \pm n^{3/4}$ and
$d^-_{\ov{J}}(w) = |\ov{Y}_w| = \ov{p}_w n \pm n^{3/4}$
for all $w \in W$.
\item For any disjoint $R,R' \sub W$ of sizes $\le s$ we have
\[ \bsize{U \cap N^-_{J^K_g}(R) \cap N^-_{\ov{J}}(R')}
= |U| \prod_{w \in R} p^g_{w,K}
\prod_{w \in R'} \ov{p}_w \pm 3sn^{3/4}. \]
\item Consider $H := 
\sum \bracc{ h(w): w \in N^+_{J^K_g}(S) \cap N^+_{\ov{J}}(S') }$
for disjoint $S,S' \sub V$ of sizes $\le s$.
\begin{align*}
& \text{If } S \cup S' \text{ is } 3d\text{-separated then } 
H = \sum_{w \in W} (p^g_{w,K})^{|S|} \ov{p}_w^{|S'|} h(w)
 \pm 5sn^{3/4}. \\
& \text{If } (S,S') \text{ is } 3d\text{-separated then } 
 H \ge 2^{-2s} \sum_{w \in W} (p^g_{w,K})^{|S|} h(w).
\end{align*}
\end{enumerate}
\end{lemma}

Write $X^g_w = \bigcup \mc{X}^g_w$
and $\ov{X}_w = V \sm (X^1_w \cup X^2_w
 \cup (X^1_w)^+ \cup (X^2_w)^+ )$.
In the proof we repeatedly use the observation that
if $S \cup S' \sub V$ is $3d$-separated and $w \in W$,
given $i(w)$ and $j(w)$,
the events $\{ \{x \in X^g_w\}: x \in S \}
\cup \{ \{x \in \ov{X}_w\}: x \in S' \}$ are independent,
as they are determined by disjoint sets 
of random decisions in INTERVALS.
The weaker assumption that $(S,S')$ is $3d$-separated 
only implies independence of $\{S \sub X^g_w\}$
and $\{S' \sub \ov{X}_w\}$. We also note that for any $S,S'$
the events $\{S \sub X^g_w \} \cap \{S' \sub \ov{X}_w\}$
are independent over $w \in W$.

\begin{proof}
For (i), consider any $I \in \mc{I}^i_j$
with $i \in [2s+1]$, $j \in [d_i]$.
For each $w \in W_i$ independently 
we have $\mb{P}(j(w)=j)=1/d_i$,
$\mb{P}(I \in \mc{S}_w \mid j(w)=j)=1/8$,
$\mb{P}(I \in \mc{X}^g_w \mid I \in \mc{S}_w)=8p^g_{w,K}$,
so $\mb{P}(I \in \mc{X}^g_w) = p^g_{w,K}/d_i$.
As $\mb{P}(w \in W_i) = 1/(2s+1)$ for each $w \in W$
and $\sum_{w \in W} p^g_{w,K} = |W|p^g_K$,
by a Chernoff bound, whp 
$|\mc{X}^g(I)| = \tfrac{|W| p^g_K}{(2s+1)d_i} \pm n^{.51}$.
This estimate holds for all such $I$, and so for
$t^g_i = \min \{ |\mc{X}^g(I)|: I \in \mc{I}^i \}$;
thus (i) holds.

For (ii), note that each $x \in V$ 
appears in exactly one interval in each $\mc{I}^i_j$,
so
$$
d^+_{J^K_g}(x,W) = \sum_{i=1}^{2s+1} \sum_{j=1}^{d_i} 
\big( \tfrac{|W| p^g_K}{(2s+1)d_i} \pm n^{.51} \big)
= |W| p^g_K \pm n^{.52}.
$$
Next we recall that 
INTERVALS chooses uniformly at random
$\mc{Y}^g(I) \sub \mc{X}^g(I)$ of size $t^g_i$.
The statements on $t^\pm_g(x)$ hold as 
for each $i$ there is exactly one $[a,b] \in \mc{I}^i$ with $a=x$ 
and exactly one $[a,b] \in \mc{I}^i$ with $b^+=x$. 
For future reference, we note that each
$|\mc{X}^g(I) \sm \mc{Y}^g(I)| < 2n^{.51}$.

For (iii), consider any $w \in W$.
We start INTERVALS by choosing $i=i(w) \in [2s+1]$
and $j=j(w) \in [d_i]$ uniformly at random.
Given these choices, any $I \in \mc{I}^i_j$ appears in $\mc{S}_w$
if $I \in \mc{A}_w$ and $I^\pm \notin \mc{A}_w$; 
this occurs with probability $1/8$, so 
$\mb{E}|\mc{S}_w|=|\mc{I}^i_j|/8 = n/8d_i \pm 1$.
As $|\mc{S}_w|$ is a $3$-Lipschitz function of
the events $\{I \in \mc{A}_w\}$, $I \in \mc{I}^i_j$,
by Lemma \ref{azuma} whp $|\mc{S}_w| = n/8d_i \pm n^{.51}$.
Each $I \in \mc{S}_w$ is included in $\mc{X}^g_w$
independently with probability $8p^g_{w,K}$,
so by a Chernoff bound whp
$|\mc{X}^g_w| = p^g_{w,K}n/d_i \pm 2n^{.51}$.
For each $I \in \mc{X}^g_w$ independently we have 
$I \in \mc{Y}^g_w$ with probability 
$t^g_i / |\mc{X}^g(I)| = 1 \pm n^{-.27}$,
as $p^g_K \ge n^{-.2}$. Thus $d_i \mb{E}|\mc{Y}^g_w| 
= p^g_{w,K}n \pm n^{.73}$, so by a Chernoff bound whp
$d^-_{J^K_g}(w) = |Y^g_w| = d_i|\mc{Y}^g_w| \pm d_i
= p^g_{w,K}n \pm 2n^{.73}$. We deduce
$d^-_{\ov{J}}(w) = n - \tfrac{d_i+1}{d_i} (|Y^1_w|+|Y^2_w|)
= \ov{p}_w n \pm n^{3/4}$, so (ii) holds.
We note that each $|Y^g_w| = |X^g_w| \pm n^{3/4}$
and $|\ov{Y}_w|=|\ov{X}_w| \pm n^{3/4}$.

For (iv), we first estimate the number $N$ of $u \in U$ 
such that $u \in X^g_w$ for all $w \in R$ and
$u \in \ov{X}_w$ for all $w \in R'$.
The actual quantity we need to estimate is obtained
by replacing `X' with `Y', and so differs in size
by at most $2sn^{3/4}$. 
For each $u \in U$, we have independently 
$\mb{P}(u \in X^g_w) = p^g_{w,K}$ for all $w \in R$ 
and $\mb{P}(u \in \ov{X}_w) = \ov{p}_w$ for all $w \in R'$,
so $\mb{E}N = |U| \prod_{w \in R} p^g_{w,K}
\prod_{w \in R'} \ov{p}_w$.
Indeed, given choices of $i=i(w)$ and $j=j(w)$, 
letting $I$ be the unique interval in $\mc{I}_j^i$
whose successor is $u$, we have
$\mb{P}(u \in \ov{X}_w) = 1-
\sum_{g = 1}^{2}(\mb{P}(u \in X_w^g) + \mb{P}(I \in \mc{X}_w^g))
= \ov{p}_w$.
Now (iv) follows from Lemma \ref{azuma},
as $N$ is a $3d$-Lipschitz function of $\le 2n$ 
independent random decisions in INTERVALS.

For (v), we will estimate 
$H' = \sum \{ h(w) : S \sub X^g_w, S' \sub \ov{X}_w \}$.
The actual quantity $H$ we need to estimate
is obtained from $H'$ by replacing `X' with `Y'.
We have $|H-H'| < 4sn^{3/4}$, as for each $i,j$ there are
$\le 2s$ intervals $I \in \mc{I}^i_j$
with $I \cap (S \cup S') \ne \es$
each with $<2n^{.51}$ choices of
$w \in \mc{X}^g(I) \sm \mc{Y}^g(I)$
each with $|h(w)| < n^{.01}$.
If $S \cup S'$ is $3d$-separated then 
independently for all $w \in W$ we have 
$\mb{P}(x \in X^g_w) = p^g_{w,K}$ for all $x \in S$ 
and $\mb{P}(x \in \ov{X}_w) = \ov{p}_w$ for all $x \in S'$;
the required estimates on $H'$ and so $H$
follow whp from Lemma \ref{bernstein}.

Finally, we consider (v) when $(S,S')$ is $3d$-separated.
We fix $w \in W$, condition on $i(w)=i$ and $j(w)=j$,
and recall $\mb{P}(S \sub X^g_w, S' \sub \ov{X}_w)
= \mb{P}(S \sub X^g_w) \mb{P}(S' \sub \ov{X}_w)$.
We have the bound
$\mb{P}(S' \sub \ov{X}_w) \ge 2^{-s}$
from the event $I \notin \mc{A}_w$ 
for all $I \in \mc{I}^i_j$ with $I \cap S' \ne \es$.
We claim that  
$\mb{P}(S \sub X^g_w) > (5s)^{-1} (p^g_{w,K})^{|S|}$,
which by Lemma \ref{bernstein} suffices to complete the proof.

To prove the claim, we first note that if for some $\mc{I}^i_j$
no two vertices of $S$ lie in consecutive intervals then 
$\mb{P}(S \sub X^g_w \mid i(w)=i, j(w)=j)
\ge (p^g_{w,K})^{|S|}$: indeed, the events 
$\{I \in \mc{X}^g_w\}$ for $I \in \mc{I}^i_j$ 
with $I \cap S \ne \es$ are positively correlated.
For $i \in [2s+1]$ let $J^i_s$ be the set of $j \in [d_i]$
for which some pair $x,x'$ of $S$ lie in consecutive intervals
of $\mc{I}^i_j$: we say $j$ is $i$-bad for $x,x'$.
We note that if $j$ is $i$-bad for some pair in $S$
then it is $i$-bad for some consecutive pair $x,x'$ in $S$
(i.e.\ $\{x,x'\} \cap S = \es$).
It suffices to show that some $|J^i_s| < d_i/2$.
For this, we note that as $|S| \le s$ we can fix $i \in [2s+1]$
so that the cyclic distance between any pair of vertices 
in $S$ is either $< d_{i+1}$ or $\ge d_{i-1}$.
There are no $i$-bad $j$ for any pair $x,x'$
with $d(x,x') \ge d_{i-1} = 2sd_i$.
Also, if $d(x,x')<d_{i+1}$ then $j$ is $i$-bad for $x,x'$
only if $\mc{I}^i_j$ contains an interval with an endpoint
in the cyclic interval $[x,x']$, so there are at most
$d_{i+1}$ such $j$. We deduce $|J^i_s| < sd_{i+1} = d_i/2$,
which completes the proof of the claim, and so of the lemma.
\end{proof}

The next lemma contains similar statements to those in
the previous one concerning the colours and directions
introduced in DIGRAPH.
In (iii) we define $J^{K'}_g$ by $J^{K'}_g[V,W]=J^K_g[V,W]$ and
$\ova{uv} \in J^{K'}_g[V] \Lra \ova{uv}^- \in J^K_g[V]$,
thus removing the twist: if for some arc $\ova{uv}$ of $G_g$
we add $\ova{uv}^-$ to $J^K_g$
then we add $\ova{uv}$ to $J^{K'}_g$.

\begin{lemma} \label{deg}
Let $g \in [2]$. 
Write $q^g_0=p^g_*$, $q^g_{K'}=p^g_K$
and $q^g_c=0$ otherwise. Then whp:
\begin{enumerate}
\item For every $v \in V$
and $c \in [3,K] \cup \{0\}$ we have
$d^\pm_{J_g}(v,V) = p_g (1 \pm \eps)\aA n \pm n^{.6}$,
$d^\pm_{J^c_g}(v,V) = p^g_c (1 \pm \eps) \aA n \pm n^{.6}$,
$d^+_{J^c_g}(v,W) = p^g_c \aA n \pm 2n^{3/4}$.
\item For every $w \in W$
and $c \in [3,K] \cup \{0\}$ we have
$d^-_{J^c_g}(w,V) = p^g_{w,c} n \pm 2n^{3/4}$.
\item For any mutually disjoint sets $R_c \sub W$ 
and $S^+_c, S^-_c \sub V$
for $c \in [3,K-1] \cup \{0,K'\}$
with $\sum_c |R_c| \le s$ 
and $\sum_c |S^\pm_c| \le s$ we have
\begin{align*}  &\Big| \bigcap_c 
\big( N^-_{J^c_g}(R_c) \cap 
N^+_{J^c_g}(S^+_c) \cap 
N^-_{J^c_g}(S^-_c) \big) \Big|\\
&= |N_G^+(\cup_c S_c^+) \cap N_G^-(\cup_c S_c^-)|
 \prod_c \Big( (q^g_c)^{|S_c^+|+|S_c^-|}
 \prod_{w \in R_c} p^g_{w,c} \Big)
\pm 4sn^{3/4}.
\end{align*}
\item Consider $H' := 
\big| W \cap N^+_{J^K_g}(S) \cap 
 \bigcap_c N^+_{J^c_g}(S_c) \big|$
for disjoint $S,S' \sub V$ of sizes $\le s$ with 
$S'$ partitioned as $(S_c: c \in [3,K-1] \cup \{0\})$.
\begin{align*}
& \text{If } S \cup S' \text{ is } 3d\text{-separated then } 
H' = \sum_{w \in W} (p^g_{w,K})^{|S|} \prod_c (p^g_{w,c})^{|S_c|} 
 \pm 6sn^{3/4}. \\
& \text{If } (S,S') \text{ is } 3d\text{-separated then } 
H' + n^{.6} \ge 2^{-2s} \sum_{w \in W} (p^g_{w,K})^{|S|} 
\prod_c (p^g_{w,c})^{|S_c|}.
\end{align*}
\end{enumerate}
\end{lemma}

\begin{proof}
All quantities considered are $1$-Lipschitz functions
of the random choices in DIGRAPH, so by Lemma \ref{azuma}
it suffices to estimate the expectations.
For (i), we recall that $G$ has vertex in- and outdegrees
$(1 \pm \eps) \aA n$, and for each $\ova{xy}$ in $G$
we have $\mb{P}(\ova{xy} \in J_g) = p_g$,
so $\mb{E}d^+_{J_g}(v,V) = p_g (1 \pm \eps)\aA n$.
The other expectations are similar, with slightly
modified calculations due to the twisting in colour $K$
and avoiding loops; for example,
$\mb{E}d^-_{J^K_g}(v,V) = p^g_K (d^-_G(v^+) \pm 1)
= p^g_K (1 \pm \eps)\aA n \pm 1$.
For (ii), we recall $d^-_{\ov{J}}(w) = \ov{p}_w n \pm n^{3/4}$
from Lemma \ref{lem:int}.iii, so for $c \ne K$ we have
$\mb{E}d^-_{J^K_c}(w) = p^g_{w,c}\ov{p}_w^{-1} d^-_{\ov{J}}(w)
= p^g_{w,c} n \pm n^{3/4}$. (The estimate for $c=K$
was already given in Lemma \ref{lem:int}.iii.)
For (iii), we first apply Lemma \ref{lem:int}.iv
with $U = N_G^+(\cup_c S_c^+) \cap N_G^-(\cup_c S_c^-)$, 
$R = R_K$
and $R' = \cup_{c \ne K} R_c$ to obtain
\begin{align*} &\phantom{=}\bsize{N_G^+(\cup_c S_c^+) \cap N_G^-(\cup_c S_c^-) \cap N^-_{J^K_g}(R_K) 
\cap N^-_{\ov{J}}(\cup_{c \ne K} R_c)}\\
&= |N_G^+(\cup_c S_c^+) \cap N_G^-(\cup_c S_c^-)| \prod_{w \in R_K} p^g_{w,K}
\prod_{w \in \cup_{c \ne K} R_c} 
\ov{p}_w \pm 3sn^{3/4}. \end{align*}
For each vertex $v$ counted here independently we have
$\mb{P}(\ova{vw} \in J^c_g \mid \ova{vw} \in \ov{J}) 
= p^g_{w,c}/\ov{p}_w$ for all $w \in R_c$,
$\mb{P}(\ova{vx} \in J^c_g \mid \ova{vx} \in G) 
= q^g_{c}$ for all $x \in S_c^-$ and
$\mb{P}(\ova{xv} \in J^c_g \mid \ova{xv} \in G) 
= q^g_{c}$ for all $x \in S_c^+$,
so whp the stated bound for (iii) holds.
For (iv) we first consider 
$H := |N^+_{J^K_g}(S) \cap N^+_{\ov{J}}(S')|$.
By Lemma \ref{lem:int}.v with $h(w)=1$,
if $S \cup S'$ is  $3d$-separated then 
$H = \sum_{w \in W} (p^g_{w,K})^{|S|} \ov{p}_w^{|S'|}
 \pm 5sn^{3/4}$, and if $(S,S')$ is $3d$-separated then
$H \ge 2^{-2s} \sum_{w \in W} (p^g_{w,K})^{|S|}$.
For each vertex $w$ counted here independently we have
$\mb{P}(\ova{vw} \in J^c_g \mid \ova{vw} \in \ov{J}) 
= p^g_{w,c}/\ov{p}_w$ for all $v \in S_c$, 
so whp the stated bound for (iv) holds.
\end{proof}

\section{Analysis II: wheel regularity} \label{sec:reg}

In this section we show how to assign weights to wheels in
each $J_g$ so that for any arc $\ova{e}$ there is total weight
about $1$ on wheels containing $\ova{e}$, and furthermore
all weights on wheels with $c+1$ vertices are of order $n^{1-c}$.
This regularity property is an assumption in the wheel
decomposition results of section \ref{sec:wheel},
and is also sufficient in its own right for approximate 
decompositions by a result of Kahn \cite{KaLP}.
The estimate for the total weight of wheels on an arc will hold
even if we add any new arc to $J_g$, which is useful 
as we will need to consider small perturbations of $J_1$
due to arcs of $G$ not allocated to $G_1$ or $G_2$ or not
covered in the approximate decomposition of $G_2$.

We start by considering wheels $\wC$ with $c<K$. Let 
\[W^g_{w,c} = n^c p^g_{w,c} (p^g_{w,0})^{c-1} (\aA p^g_*)^c.\]
The motivation for this formula is that it is about the
expected number of $\wC$'s in $J_g$ using $w$. 
For any arc $\ova{e}$ let $W^g_c(\ova{e})$ 
be the set of copies of $\wC$ in $J_g$ with hub in $W$ using $\ova{e}$.
Let \[ \hat{W}^g_c(\ova{e}) = \sum \{ 
p^g_{w,c} n (W^g_{w,c})^{-1} : 
\mc{W} \in W^g_c(\ova{e}), w \in V(\mc{W}) \}.\]
(If $p^g_{w,c}=0$ there are no such $\mc{W}$,
so $(W^g_{w,c})^{-1}$ is always defined when used.)
In the following lemma we calculate the total weights on arcs
due to copies of $\wC$, although we note that we do not have
a good estimate for $\ova{xy} \in J^0_g[V]$ if $d(x,y)<3d$.
In $J_2$ we can ignore such arcs, as we only need an 
approximate decomposition, whereas in $J_1$ we will cover
these by wheels greedily before finding the exact decomposition
-- this forms part of the perturbation referred to above.

\begin{lemma} \label{degWc}
Let $c' \in \{0,c\}$, $N_c=1$ and $N_0 = c-1$.
Then whp:
\begin{enumerate}
\item If $p_{w,c'}^g \neq 0$ and we add $\ova{xw}$ to $J^{c'}_g[V,W]$ 
then $\hat{W}^g_c(\ova{xw})
= (1 \pm 4\eps) N_{c'} p^g_{w,c}/p^g_{w,c'} \pm n^{-.2}$.
\item If $d(x,y)\ge 3d$ and  we add $\ova{xy}$ to $J^0_g[V]$ 
then $\hat{W}^g_c(\ova{xy}) 
= (1 \pm 4\eps) cp^g_c/p^g_* \pm n^{-.2}$.
\end{enumerate}
\end{lemma}

\begin{proof}
As a preliminary step for counting copies of $\wC$
we count $c$-prewheels, which we define to consist of a wheel
with oriented rim cycle in $G$ and all spokes in $\ov{J}$.
For any arc $\ova{e}$ we let $P_c(\ova{e})$ be the set of $c$-prewheels 
using $\ova{e}$; we will estimate $|P_c(\ova{e})|$ using the analysis 
of INTERVALS in Lemma \ref{lem:int}.
 
For (i), we estimate $|P_c(\ova{xw})|$ as follows.
We let $x=x_c$ and choose the other rim vertices 
$x_1,\dots,x_{c-1}$ sequentially in cyclic order.
At $c-2$ steps we choose 
$x_{i+1} \in N_G^+(x_i) \cap N^-_{\ov{J}}(w)$:
each has $\aA n \ov{p}_w \pm 3sn^{3/4}$ 
options by Lemma \ref{lem:int}.iv 
with $U=N_G^+(x_i)$, $R=\es$, $R'=\{w\}$,
using $|N_G^+(x_i)|=\aA n$ ($G$ is $\aA n$-regular).
At the last step we choose
$x_{c-1} \in N_G^+(x_{c-2}) \cap N_G^-(x_c) \cap 
N^-_{\ov{J}}(w)$, so similarly there are 
$|N_G^+(x_{c-2}) \cap N_G^-(x_c)| \ov{p}_w \pm 3sn^{3/4}$ 
options, where $|N_G^+(x_{c-2}) \cap N_G^-(x_c)|
= ((1 \pm \eps)\aA)^2 n$ by typicality of $G$.
Thus $|P_c(\ova{xw})| 
= (1 \pm 3\eps) \aA^c (\ov{p}_w n)^{c-1}$.

Now consider the case $c'=c$,
i.e.\ $\ova{xw}$ is added to $J^c[V,W]$.
For any $c$-prewheel containing $\ova{xw}$, 
independently we include the cycle arcs in $J^0_g$ 
with probability $p^g_*$ 
and give each $\ova{x_i w}$ with $i \ne c$ colour $0$ 
with probability $p^g_{w,0}/\ov{p}_w$, so
$\mb{E}|W^g_c(\ova{xw})| 
= (1 \pm 3\eps) (\aA p^g_*)^c (p^g_{w,0} n)^{c-1} 
= (1 \pm 3\eps) W^g_{w,c}/p^g_{w,c}n$.
Of these random decisions, 
$\le 2n$ concern an arc containing one of $x,w$, 
which affect $|W^g_c(\ova{xw})|$ by $O(n^{c-2})$,
and the others have effect $O(n^{c-3})$.
Thus $|W^g_c(\ova{xw})|$ is $O(n^{2c-3})$-varying,
so by Lemma \ref{azuma} whp $|W^g_c(\ova{xw})| 
= (1 \pm 4\eps) W^g_{w,c}/p^g_{w,c}n$,
i.e.\ $\hat{W}^g_c(\ova{xw}) = 1 \pm 4\eps$.
When $c'=0$ we argue similarly.
Now $x$ can be any $x_i$ with $i \ne c$,
for which we have $c-1$ choices.
The probability factors are the same as in the previous
calculation, except that for $\ova{x_c w}$ we replace
$p^g_{w,0}/\ov{p}_w$ by $p^g_{w,c}/\ov{p}_w$.
Again, the stated estimate holds whp 
by Lemma \ref{azuma}, so (i) holds.

For (ii), we write $\hat{W}^g_c(\ova{xy}) 
= \sum_{w \in W} \hat{W}^g_c(xyw)$, 
where $\hat{W}^g_c(xyw)$ is the sum of $(W^g_{w,c})^{-1}$
over the set $W^g_c(xyw)$ of copies of $\wC$ in $J_g$ 
using $\ova{xy}$, $\ova{xw}$ and $\ova{yw}$. 
Fix $w \in N^+_{\ov{J}}(x) \cap N^+_{\ov{J}}(y)$ and
consider the number $|P_c(xyw)|$ of $c$-prewheels
using $\{\ova{xy},\ova{xw},\ova{yw}\}$. Choosing rim
vertices sequentially as in (i), now there are
$c-3$ steps with $\aA n \ov{p}_w \pm 3sn^{3/4}$ options
and again $((1 \pm \eps)\aA)^2 \ov{p}_w n \pm 3sn^{3/4}$
options at the last step, so $|P_c(xyw)| 
= (1 \pm 3\eps) \aA^{c-1} (\ov{p}_w n)^{c-2}$.

Now we consider which of these $c$-prewheels 
extend to wheels in $W^g_c(xyw)$:
there are $c$ choices for the position of $\ova{xy}$ on the rim,
then some probabilities determined by independent random decisions:
the $c-1$ rim edges are each correct with probability $p^g_*$,
the spoke of colour $c$ with probability $p^g_{w,c}/\ov{p}_w$,
and the other $c-1$ spokes each with probability $p^g_{w,0}/\ov{p}_w$.
Therefore \[ \mb{E} \hat{W}^g_c(xyw)
= (1 \pm 3\eps) c (\aA p^g_*)^{c-1} p^g_{w,c} (p^g_{w,0})^{c-1} 
\ov{p}_w^{-2} n^{c-2} p^g_{w,c} n (W^g_{w,c})^{-1}
= (1 \pm 3\eps) c (\aA p^g_*)^{-1} p^g_{w,c} n (\ov{p}_w n)^{-2}. \]
By Lemma \ref{azuma} whp $\hat{W}^g_c(\ova{xy}) 
= (1 \pm 3.1\eps) c (\aA p^g_* n)^{-1} H$, 
with $H = \sum \{ p^g_{w,c} \ov{p}_w^{-2} : 
w \in N^+_{\ov{J}}(x) \cap N^+_{\ov{J}}(y) \}$.

We estimate $H$ by Lemma \ref{lem:int}.v with $S=\es$, $S'=\{x,y\}$
and $h(w) = p^g_{w,c}\ov{p}_w^{-2}$ (each $7/8 \le \ov{p}_w \le 1$).
As $S \cup S'$ is $3d$-separated,
whp $H = |W|p^g_c \pm 5sn^{3/4}$, giving
$\hat{W}^g_c(\ova{xy}) = (1 \pm 4\eps) cp^g_c/p^g_* \pm n^{-.2}$.
\end{proof}

Now we apply a similar analysis for $\wK$. Let 
\[ W^g_{w,K} = n^8  \aA p^g_K p^g_{w,K} 
(\aA p^g_* p^g_{w,0})^7 .\]
For any arc $\ova{e}$ let $W^g_K(\ova{e})$ 
be the set of copies of $\wK$ in $J_g$ using $\ova{e}$.
We define $\hat{W}^g_K(\ova{e})$ by
setting $c=K$ in $\hat{W}^g_c(\ova{e})$. 
Now we calculate the total weights on arcs
due to copies of $\wK$. Note that we cannot give
a good estimate for $\ova{xy} \in J^K_g[V]$ if $d(x,y)<3d$.
We can ignore such arcs in $J_2$ (as mentioned above), 
but in $J_1$ we will replace such arcs 
by arcs of colour $0$ (modified by twisting)
-- this also forms part of the perturbation.

\begin{lemma} \label{degWK}
Let $c' \in \{0,K\}$, $N_K=1$, 
$N_0 = 7$, $q^g_K = p^g_K$, $q^g_0=p^g_*$. Then whp:
\begin{enumerate}
\item If we add $\ova{xw}$ to 
$J^{c'}_g[V,W]$ then $\hat{W}^g_K(\ova{xw})
= (1 \pm 4\eps) N_{c'} p^g_{w,K}/p^g_{w,c'}$.
\item Suppose we add $\ova{xy}$ to $J^{c'}_g[V]$. 
If $d(x,y)\ge 3d$ then $\hat{W}^g_c(\ova{xy}) 
= (1 \pm 4\eps) N_{c'} q^g_K/q^g_{c'}$.\\
If $c'=0$ then
$\hat{W}^g_K(\ova{xy}) > 2^{-2s-1} p^g_K/p^g_*$.
\end{enumerate}
\end{lemma}

\begin{proof}
For (i), we start by counting $(K,g)$-prewheels, which 
we define to consist of a hub $w \in W$ and an oriented $8$-path 
in $G$ between $z$ and $z^+$ for some $z$ such that
$\ova{zw} \in J^K_g$ and $\ova{z'w} \in \ov{J}$
for all internal vertices $z'$ of the path.
For any arc $\ova{e}$ we let $P^g_K(\ova{e})$ be the set of 
$(K,g)$-prewheels using $\ova{e}$.

To estimate $|P^g_K(\ova{xw})|$, suppose first that $c'=K$.
We require $z=x$. We choose the vertices of the path one by one.
At $6$ steps there are $\aA n \ov{p}_w \pm 3sn^{3/4}$ options,
and at the last step 
$((1 \pm \eps)\aA)^2 \ov{p}_w n \pm 3sn^{3/4}$
options of a common outneighbour of some vertex and $z^+$, so 
$|P^g_K(\ova{xw})| = (1 \pm 3\eps) \aA^8 (\ov{p}_w n)^7$.
On the other hand, if $c'=0$ then there are $7$ choices
for the position of $x$ as an internal vertex,
dividing the path into two segments.
We construct one segment by choosing its vertices one by one,
and then do the same for the other segment, starting with one of 
length $\le 4$ so that $\{z,z^+\}$ is not the last choice. 
At the step where we choose $\{z,z^+\}$, there is some vertex $v$ 
on the path for which we need the arc $\ova{vz}$ or $\ova{vz}^+$.
We also require $z \in N^-_{J^K_g}(w)$.
The number of options is $\aA n p^g_{w,K} \pm 3sn^{3/4}$ 
by Lemma \ref{lem:int}.iv, with 
$R=\{w\}$, $R=\es$ and $U=N_G^+(v)$ 
or $U=N_G^+(v)^- = \{z: \ova{vz}^+ \in G\}$.
There are also $5$ steps with
$\aA n \ov{p}_w \pm 3sn^{3/4}$ options,
and at the last step 
$((1 \pm \eps)\aA)^2 \ov{p}_w n \pm 3sn^{3/4}$ options, 
so $|P^g_K(\ova{xw})| = (1 \pm 3\eps) 
7 \aA^8 p^g_{w,K} (\ov{p}_w)^6 n^7$
(as $p^g_{w,K} \ge n^{-.2}/8$).

To estimate $|W^g_K(\ova{xw})|$, we first consider $c'=K$.
For any $(K,g)$-prewheel containing $\ova{xw}$, 
independently we include the last path arc (to $z^+$)
in $J^K_g$ with probability $p^g_K$,
the other $7$ path arcs in $J^0_g$ 
with probability $p^g_*$,
and give $\lova{wz}'$ for each internal vertex $z'$ colour $0$ 
with probability $p^g_{w,0}/\ov{p}_w$, so
$$
\mb{E}|W^g_K(\ova{xw})| 
= (1 \pm 3\eps) \aA p^g_K (\aA p^g_* p^g_{w,0} n)^7
= (1 \pm 3\eps) W^g_{w,K}/p^g_{w,K}n.
$$
As $|W^g_K(\ova{xw})|$ is $O(n^{13})$-varying,
by Lemma \ref{azuma} whp $|W^g_K(\ova{xw})| 
= (1 \pm 3.1\eps) W^g_{w,K}/p^g_{w,K}n \pm n^{6.51}$,
so $\hat{W}^g_K(\ova{xw}) = 1 \pm 4\eps$
(using $p^g_K>n^{-.2}$).

For $c'=0$ we have a similar calculation.
Indeed, the path arcs are again correct with
probability $(p^g_*)^7 p^g_K$,
and the arcs $\lova{wz}'$ (now excluding $z'=x$)
are correct with probability $(p^g_{w,0}/\ov{p}_w)^6$,
so
$$
\mb{E}|W^g_K(\ova{xw})| 
= (1 \pm 3\eps) 7 \aA p^g_K p^g_{w,K} 
(p^g_{w,0})^6 (\aA p^g_* n)^7
= (1 \pm 3\eps) 7 W^g_{w,K}/p^g_{w,0}n.
$$
By Lemma \ref{azuma} whp
$|W^g_K(\ova{xw})| = (1 \pm 4\eps) 7W^g_{w,K}/p^g_{w,0}n 
\pm n^{6.51}$, so $\hat{W}^g_K(\ova{xw}) 
= (1 \pm 4\eps) 7p^g_{w,K}/p^g_{w,0}$.

For (ii), we write $\hat{W}^g_K(\ova{xy}) 
= \sum_{w \in W} |\hat{W}^g_K(xyw)|$, 
where $\hat{W}^g_K(xyw)$ is the sum of $(W^g_{w,K})^{-1}$
over the set $W^g_K(xyw)$ of copies of $\wK$ in $J_g$ 
using $\ova{xy}$, $\ova{xw}$ and $\ova{yw}$. For each $w$
we consider the set  $P^g_K(xyw)$ of $(K,g)$-prewheels
using $\{\ova{xy},\ova{xw},\ova{yw}\}$

Suppose first that $\ova{xy}$ has colour $c'=K$. 
We assume $d(x,y) \ge 3d$ (or there is nothing to prove).
We must have $y=z$ and in our prewheels the oriented $8$-paths 
from $z$ to $z^+$ must end with the arc $\ova{xz}^+$,  
corresponding to $\ova{xy} \in J^K$ under twisting.
We need $w \in N^+_{\ov{J}}(x) \cap N^+_{J^K_g}(y)$
so that $\ova{yw}$ has colour $K$ and $\ova{xw}$ can
receive colour $0$. Choosing rim vertices sequentially,
now $\{z,z'\}$ is already fixed, there are
$5$ steps with $\aA n \ov{p}_w \pm 3sn^{3/4}$ options,
and at the last step 
$((1 \pm \eps)\aA)^2 \ov{p}_w n \pm 3sn^{3/4}$ options, 
so $|P^g_K(\ova{xw})| = (1 \pm 3\eps) 
 \aA^7  (\ov{p}_w)^6 n^6$.

Now consider which of these prewheels extend to wheels in $W^g_K(xyw)$, 
according to the following independent random decisions:
the other $7$ arcs of the oriented $8$-path excluding $\ova{xy}$
are each correct with probability $p^g_*$,
we already have $\ova{yw} \in J^K_g$,
and for each of the $7$ internal vertices $z'$ we have
$\ova{z'w}$ correct with probability $p^g_{w,0}/\ov{p}_w$.
Therefore \[ \mb{E} \hat{W}^g_K(xyw)
= (1 \pm 3\eps)  (\aA p^g_*)^7  (p^g_{w,0})^7 
\ov{p}_w^{-1} n^6 p^g_{w,K} n (W^g_{w,K})^{-1}
= (1 \pm 3\eps) (\aA p^g_K \ov{p}_w n )^{-1}. \]
By Lemma \ref{azuma} whp $\hat{W}^g_K(\ova{xy}) 
= (1 \pm 3.1\eps) (\aA p^g_K n)^{-1}  H \pm n^{-.2}$, 
with $H = \sum \{ \ov{p}_w^{-1}  : 
w \in N^+_{\ov{J}}(x) \cap N^+_{J^K_g}(y) \}$.
We estimate $H$ by Lemma \ref{lem:int}.v 
with $S=\{y\}$ and $S'=\{x\}$. As $d(x,y) \ge 3d$,
whp $H = |W| p^g_K \pm 5sn^{3/4}$, giving
$\hat{W}^g_K(\ova{xy}) = 1 \pm 4\eps$.

Now suppose that $\ova{xy}$ has colour $c'=0$. 
For the hub $w$ we require $\ova{yw} \in J^0$ 
and $\ova{xw}$ in $J^K$ or $J^0$.
We first consider the contribution
from $\ova{xw} \in J^K$, when the 
first vertex of the oriented $8$-path must be $z=x$.
The estimate of $|P^g_K(\ova{xw})|$ 
is the same as when $c'=K$, and the probability factors 
are the same except that the factor for the last path
edge (to $z^+$) is now $p^g_K$ instead of $p^g_*$.
If $d(x,y) \ge 3d$ then the same calculation with
Lemma \ref{azuma} and Lemma \ref{lem:int}.v 
shows that the contribution to $\hat{W}^g_K(\ova{xy})$
from $w \in N^+_{\ov{J}}(x) \cap N^+_{J^K_g}(y) \}$
is $(1 \pm 4\eps) (p^g_* n)^{-1}$.

Now we consider the contribution from $\ova{xw} \in J^0$.
There are $6$ positions for $\ova{xy}$ on the path 
avoiding $\{z,z'\}$. The estimate of $|P^g_K(\ova{xw})|$ 
is the same as before except that one factor of $\ov{p}_w$ 
is replaced by $p^g_{w,K}$ (at the choice of $\{z,z'\}$).
The probability factors are the same as in the previous
calculation for $\ova{xw} \in J^K$, so 
$\mb{E} \hat{W}^g_K(xyw) = (1 \pm 3\eps) p^g_{w,K} 
(\aA p^g_* \ov{p}_w^2 n )^{-1}$.
By Lemma \ref{azuma} whp the contribution 
to $\hat{W}^g_K(\ova{xy})$ from such $w$ is
$(1 \pm 3.1\eps) 6 (\aA p^g_* n)^{-1} H$, 
with $H = \sum \{ h(w)  : 
w \in N^+_{\ov{J}}(x) \cap N^+_{\ov{J}}(y) \}$,
$h(w) = p^g_{w,K} (\ov{p}_w)^{-2}$. 

We estimate $H$ by Lemma \ref{lem:int}.v with $S=\es$, $S'=\{x,y\}$.
As $(S,S')$ is $3d$-separated (vacuously) whp
$H \ge 2^{-2s} \sum_{w \in W} h(w) = 2^{-2s} |W|p^g_K$,
so $\hat{W}^g_K(\ova{xy}) > 2^{-2s-1} p^g_K/p^g_*$.
Now suppose $d(x,y) \ge 3d$. Then $S \cup S'$ 
is $3d$-separated, so whp $H = |W|p^g_K \pm 5sn^{3/4}$. 
The contribution here to $\hat{W}^g_K(\ova{xy})$ is 
$(1 \pm 4\eps) 6 p^g_K/p^g_* $, so altogether
$\hat{W}^g_K(\ova{xy}) = (1 \pm 4\eps) 7 p^g_K/p^g_*$.
\end{proof}

We combine the above estimates to deduce the main lemma
of this section, establishing wheel regularity.
Let \[ \hat{W}^g(\ova{e}) = \sum \{ 
\hat{W}^g_c(\ova{e}): c \in [3,K] \}.\]

\begin{lemma} \label{reg}
Suppose we add $\ova{e}$ to $J$ in any colour,
such that if $\ova{e} \in J[V]$
then $\ova{e}=\ova{xy}$ with $d(x,y) \ge 3d$, 
and if $\ova{e}$ has a vertex in $W$ then it is an endvertex.
Then $\hat{W}^g(\ova{e}) = 1 \pm 5\eps$.
\end{lemma} 

\begin{proof}
By Lemmas \ref{degWc} and \ref{degWK} we can
analyse the various cases as follows.
\begin{itemize}
\item
If $\ova{e} \in J^c_g[V,W]$ with $c \ne 0$
then $\hat{W}^g(\ova{e}) = \hat{W}^g_c(\ova{e}) = 1 \pm 5\eps$.
\item
If $\ova{xy} \in J^K_g[V]$ with $d(x,y) \ge 3d$ then 
$\hat{W}^g(\ova{e}) = \hat{W}^g_K(\ova{e}) = 1 \pm 5\eps$.
\item
If $\ova{e} \in J^0_g[V,W]$ then
$$
\hat{W}^g(\ova{e}) = 
(1 \pm 4\eps) 7 p^g_{w,K}/p^g_{w,0}  + \textstyle\sum_{c=3}^{K-1}
\big( (1 \pm 4\eps) (c-1) p^g_{w,c}/p^g_{w,0} \pm n^{-.2} \big)
= 1 \pm 5\eps,
$$
as
$p^g_{w,0} = 7 p^g_{w,K} +  \sum_{c=3}^{K-1} (c-1) p^g_{w,c}$.
\item
If $\ova{xy} \in J^0_g[V]$ with $d(x,y) \ge 3d$
then
$$
\hat{W}^g(\ova{e}) = (1 \pm 4\eps) 7 p^g_K/p^g_*
+ \textstyle\sum_{c=3}^{K-1} \big( 
 (1 \pm 4\eps) cp^g_c/p^g_* \pm n^{-.2} \big)
= 1 \pm 5\eps,
$$
as $p^g_* = p_g - p^g_K
= 7p^g_K + \sum_{c=3}^{K-1} cp^g_c$. \vspace{-0.5cm}
\end{itemize}
\end{proof}

\section{Approximate decomposition} \label{sec:approx}

Here we describe the approximate decomposition of $G_2$.
Recall that at the start of section \ref{sec:alg}
we partitioned each factor $F_w$ into subfactors 
$F^1_w$ and $F^2_w$, that each $F^g_w$ 
has $q^g_{w,c} n$ cycles of length $c \in [3,K-1]$,
and $p^g_{w,c} = (1-\eta)q^g_{w,c}$.
We will embed almost all of each $F^2_w$ in $G_2$.
We say $F'_w \sub F^2_w$ is valid if it does not have
any independent arcs (i.e.\ arcs $\ova{xy}$ such that 
both $x$ and $y$ have total degree $1$ in $F'_w$)
and if $F^2_w$ contains a path then $F'_w$
contains the arcs incident to each of its ends.

\begin{lemma} \label{lem:approx}
There are arc-disjoint digraphs $G^2_w \sub G_2$ for $w \in W$,
where each $G^2_w$ is a copy of some valid $F'_w \sub F^2_w$
with $V(G^2_w) \sub N^-_{J_2}(w)$, such that 
\begin{enumerate}
\item $G_2^- = G_2 \sm \bigcup_{w \in W} G^2_w$
has maximum degree at most $5d^{-1/3}n$,
\item the digraph $J_2^-$
obtained from $J_2[V,W]$ by deleting all $\ova{xw}$ 
with $x \in V(G^2_w)$ has maximum degree at most $5d^{-1/3}n$, and
\item any $x \in V$ has degree $1$ in $F'_w$
for at most $n/\sqrt{d}$ choices of $w$.
\end{enumerate}
\end{lemma}
 
\begin{proof}
Say that an arc $\ova{vw}$ with $v \in V$ and $w \in W$ is \emph{bad}
there is some $c \in [3,K-1]$ such that
$\ova{vw} \in J^c$ and $p^2_{w,c} < n^{-.1}$,
or $\ova{vw} \in J^K$ and $p^2_{w,K} < d^{-1/3}$.
The expected bad degree of $v \in V$
is at most $(Kn^{-.1}+d^{-1/3})n$
so by Chernoff bounds we can assume that
every $v \in V$ has bad degree at most $2d^{-1/3}n$.
Let $J'_2$ be obtained from $J_2$ by deleting 
all bad arcs and
all $\ova{xy} \in J^K_2[V]$ with $d(x,y)<3d$.
We consider the auxiliary hypergraph $\mc{H}$
whose vertices are all arcs of $J'_2$ and 
whose edges correspond to all copies of the 
coloured wheels $\wK$ or $\wC$ with $c \in [3,K-1]$.
We recall that
$W^g_{w,c} = n^c p^g_{w,c} (p^g_{w,0})^{c-1} (\aA p^g_*)^c$
and $W^g_{w,K} = n^8  \aA p^g_K p^g_{w,K} 
(\aA p^g_* p^g_{w,0})^7 $.
We assign weights 
$(1-5\eps)p^g_{w,c} n / (W^g_{w,c})^{-1}$ 
to each copy of any $\wC$ (and to $\wK$ for $c=K$).
By Lemma \ref{reg}, the total weight 
of wheels in $J_2$ on any arc $\ova{e}$ satisfies 
$1-10\eps < \hat{W}^g(\ova{e}) < 1$.
Thus the total weight of wheels in $J_2'$
on any arc $\ova{e}$ satisfies 
$1-d^{-1/4} < \hat{W}^g(\ova{e}) < 1$,
as we deleted at most $2d^{-1/3}n^7$ (say) copies of $\wK$
on $\ova{e}$ using a deleted arc.
Note also that for any two arcs the total weight of wheels 
containing both is at most $n^{-.7}$ (as $p^g_K \ge n^{-1/4}$).

Thus $\mc{H}$ satisfies the hypotheses of a result of Kahn \cite{KaLP} 
on almost perfect matchings in weighted hypergraphs
that are approximately vertex regular and have small codegrees.
A special case of this result (slightly modified) implies that
for any collection $\mc{F}$ of at most $n^{100}$ (say) subsets 
of $V(\mc{H})=J$ each of size at least $\sqrt{n}$ (say) we can find 
a matching $M$ in $\mc{H}$ such that 
$|F \sm \bigcup M| < d^{-1/5} |F|$ for all $F \in \mc{F}$.
(This is immediate from \cite{KaLP} if $\mc{F}$ has constant size,
and a slight modification using better concentration inequalities
implies the stated version. Alternatively, one can reduce to the
problem to an unweighted version via a suitable random selection 
of edges and then apply a result of Alon and Yuster \cite{AY}.)
This is also implied by a recent result of Ehard,
Glock and Joos~\cite{EGJ}.

We choose such a matching $M$ for the family $\mc{F}$
where for each $v \in V \cup W$ we include sets 
$F_v = \{ \ova{e} \in J_2[V,W]: v \in \ova{e} \}$,
$F^K_v = \{ \ova{e} \in J^K_2[V,W]: v \in \ova{e} \}$,
and $F'_v = \{ \ova{e} \in J_2[V]: v \in \ova{e} \}$
(the last just for $v \in V$). This $\mc{F}$ is valid as
all $|F|>\sqrt{n}$ by Lemma \ref{deg}.
By construction
for all $c \in [3,K-1]$ every copy of $\wC$ in $M$
with hub $w$ has $p^2_{w,c} \ge n^{-.1}$ and
every copy of $\wK$ in $M$
with hub $w$ has $p^2_{w,K} \ge nd^{-1/3}$.

For each $w$ we define $G^2_w$ to be the subgraph of
$G$ corresponding to the wheels in $M$ containing $w$, 
where we take account of the twisting in colour $K$.
Thus $G^2_w$ contains the rim $c$-cycle
of any $c$-wheel in $M$ containing $w$,
and for any copy of $\wK$ in $M$ containing 
$\ova{xw} \in J^K[V,W]$ we obtain 
an oriented path of length $8$ from $x$ to $x^+$.
The maximum degree bounds in (i) and (ii) clearly hold.

Recalling that $N^-_{J_2}(w)$ is disjoint from
the set of interval successors $(Y^2_w)^+$,
we see that these cycles and paths are vertex-disjoint,
except that some paths may connect up to form longer paths, 
which can be described as follows. Let $\mc{Y}'_w$ be 
the set of maximal cyclic intervals $I$ such that 
for every $x \in I$ there is a copy of $\wK$ in $M$ 
containing $\ova{xw} \in J^K[V,W]$. Then for each 
$[a,b] \in \mc{Y}'_w$ we have a component of $G^2_w$ 
that is a path of length $8d(a,b)$ from $a$ to $b^+$.
All these paths have length at most $8d$, as each such $I$ 
is contained within an interval of $\mc{Y}^2_w$.
Furthermore, if $x \in V$ is an endpoint of some path 
in $G^2_w$ then either $x$ is a startpoint or successor of 
some interval in $\mc{Y}^2_w$, for which there are 
at most $2t_2$ choices of $w$ by Lemma \ref{lem:int},
or $x^+ w \in F^K_{x^+} \sm  \bigcup M$,
or $x^- w \in F^K_{x^-} \sm  \bigcup M$,
giving at most $2n/K$ more choices of $w$,
for a total of at most $n/\sqrt{d}$ (say).

It remains to show that each $G^2_w$ is isomorphic
to some valid $F'_w \sub F_w$. First we show
for any $c \in [3,K-1]$ that whp each $G^2_w$ 
has at most $q^2_{w,c} n$ cycles of length $c$.
The number of $c$-cycles
is in $G^2_w$ is at most $|N^-_{J^c_2}(w)|$, which 
by Chernoff bounds is whp $< p^2_{w,c} n + n^{.6}
= (1-\eta) q^2_{w,c} n + n^{.6} < q^2_{w,c} n$,
recalling that $p^2_{w,c} \ge n^{-.1}$.
Next we bound the total length $L_w$ of paths in $G^2_w$.
By Lemma \ref{lem:int} we have 
$L_w \le 8|Y^2_w| < 8p^2_{w,K} n + 8n^{3/4}$.
Writing $L'_w$ for the total length of long
(length $\ge K$) cycles and paths in $F^2_w$,
we recall that $8p^2_{w,K} n
= p^2_w n - \sum_{c=3}^{K-1} cp^2_{w,c}n
= (1-\eta)(L'_w+n^{.8})$.
So since $p^2_{w,K} \ge d^{-1/3}n$, we have
$L'_w > 8d^{-1/3}n$ and $L_w < (1-\eta/2)L'_w$.

We embed the paths of $G^2_w$ into the long cycles and
paths in $F^2_w$ according to a greedy algorithm,
where in each step that we embed some path $P$ of $G^2_w$
we delete a path of length $|P|+4$ from $F^2_w$,
which we allocate to a copy of $P$ surrounded
on both sides by paths of length $2$ that we will not 
include in $F'_2$ (so that $F'_2$ will be valid).
We choose such a path (if it exists) within a remaining
cycle or path of $G^2_w$, using an endpoint if it is 
a path (so that we preserve the number of components).
Recalling that there are at most $n/\sqrt{d}$ endpoints
of paths in $G^2_w$, we thus allocate a total of at most 
$2n/\sqrt{d}$ edges to the surrounding paths of length $2$.
Suppose for a contradiction that the algorithm gets stuck,
trying to embed some path $P$ in some remainder $R$.
Then all components of $R$ have size $\le |P|+5 \le 8d+5$.
All components of $G^2_w$ have size $\ge K$,
so $|R| \le (8d+5)|L'_w|/K$. However, we also have 
$|R| \ge |L'_w|-|L_w|-2n/\sqrt{d}
\ge \eta |L'_w|/2 - 2n/\sqrt{d}$, which is a contradiction, 
as $K^{-1} \ll d^{-1} \ll \eta$ and $L'_w > 8d^{-1/3}n$.
Thus the algorithm succeeds in constructing a valid
copy $F'_w$ of $G^2_w$ in $F^2_w$.
\end{proof}

\section{Exact decomposition} \label{sec:exact}
  
This section contains the two exact decomposition results
that will conclude the proof in both Case $K$ and Case $\ell^*$.
We start by giving a common setting for both cases.
We say that $G'_1$ is a $\gamma$-perturbation of $G_1$ if 
$|N_{G_1}^\pm(x) \sd N_{G'_1}^\pm(x)| < \gamma n$ for any $x \in V$.
We say that $J'_1$ is a $\gamma$-perturbation of $J_1$ 
if $J'_1$ is obtained from $J_1$ by adding, deleting
or recolouring at most $\gamma n$ arcs at each vertex.
We will only consider perturbations which are
compatible in the sense that arcs added between $V$ and $W$
will point towards $W$,
and existing colours will be used.

\begin{set} \label{set}
Let $G'_1$ be an $\eta^{.9}$-perturbation of $G_1$.
Suppose for each $w \in W$ that $Z_w \sub V$ 
with $|Z_w \sd (V \sm N^-_{J^1}(w))| < 5\eta n$.
For $x \in V$ we write $Z(x)=\{w \in W: x \in Z_w\}$.
\end{set}

We start with the exact result for Case $\ell^*$, where we
recall that $F^1_w$ consists of exactly $L^{-3} n$ cycles of length 
$\ell^*$, so $p^1_w = (1-\eta)\ell^* L^{-3} + n^{-.2}$,
$p^1_{w,\ell^*} = (1-\eta) L^{-3}$, $p^1_{w,K} = n^{-.2}/8$
and $p^1_{w,c}=0$ for $c \in [3,K-1]$.

\begin{lemma} \label{exactL}
Suppose in Setting \ref{set} and Case $\ell^*$ that 
$d_{G'_1}^\pm(x)=|W|-|Z(x)|$ for all $x \in V$
and $\ell^*$ divides $n-|Z_w|$ for all $w \in W$.
Then $G'_1$ can be partitioned into graphs $(G^1_w: w \in W)$,
where each $G^1_w$ is an oriented $C_{\ell^*}$-factor 
with $V(G^1_w) = V \sm Z_w$.
\end{lemma}

\begin{proof}
We will show that there is a perturbation $J'_1$ of $J_1$
such that $J'_1[V] = G'_1$,
each $N^-_{J'_1}(w) = V \sm Z_w$, and Theorem \ref{decompL}
applies to give a $\wL$-decomposition of $J'_1$.
This will suffice, by taking each $G^1_w$ to consist of
the rim $\ell^*$-cycles of the copies of $\wL$ containing $w$.

We construct $J'_1$ by starting with $J'_1=J_1$ and applying
a series of modifications as follows. First we delete all arcs
of $J'_1[V]$ corresponding to arcs of $G_1 \sm G'_1$ and add 
arcs of colour $0$
corresponding to arcs of $G'_1 \sm G_1$.
Similarly, we delete all arcs $\ova{vw} \in J'_1[V,W]$ 
with $v \in N^-_{J_1}(w) \cap Z_w$ and add arcs $\ova{vw}$
of colour $0$ for each $v \in (V \sm Z_w) \sm N^-_{J_1}(w)$.
We also recolour any $\ova{vw} \in J'_1[V,W]$ of colour $K$
to have colour $0$ and replace any $\ova{xy}$ of colour $K$
in $J'_1[V]$ by $\ova{xy}^+$ of colour $0$. 
As each $p^1_{w,K}=n^{-.2}/8$ in this case,
whp this affects at most $n^{.8}$ arcs at any vertex. 
Now $J'_1[V]=G'_1$,
each $N^-_{J'_1}(w) = V \sm Z_w$ and
$J'_1$ is a $\eta^{.8}$-perturbation of $J_1$.
We note for each $x \in V$ that $d_{J'_1}^\pm(x,V) 
= d_{G'_1}^\pm(x) = |W|-|Z(x)| = d^+_{J'_1}(x,W)$,
so the divisibility conditions for $x \in V$ are satisfied.

Finally, to satisfy the divisibility conditions
for all $w \in W$ we recolour so that  
$d^-_{(J'_1)^{\ell^*}}(w) = d^-_{J'_1}(w)/\ell^*$,
which is an integer, as $\ell^*$ divides
$d^-_{J'_1}(w) = n-|Z_w|$. 
By Lemma \ref{deg} each $d^-_{J_1}(w) = p^1_w n \pm 2n^{3/4}$
and $d^-_{J_1^{\ell^*}}(w) = p^1_{w,\ell^*} n \pm 2n^{3/4}$,
where $p^1_w = \ell^* p^1_{w,\ell^*} + n^{-.2}$ in this case.
As $J'_1$ is an $\eta^{.8}$-perturbation of $J_1$,
we only need to recolour at most $2\eta^{.8} n$ arcs
at any vertex, so our final digraph $J'_1$ 
is a $3\eta^{.8}$-perturbation of $J_1$.

Next we consider the regularity condition of Theorem \ref{decompK}.
To each copy of $\wL$ in $J'_1$ with hub $w$ 
we assign weight $p^1_{w,\ell^*} n/ W^g_{w,\ell^*} = 
p^1_{w,0} n (\aA p^1_{w,0} p^1_* n)^{-\ell^*}$,
which lies in $[n^{1-\ell^*}, L^L n^{1-\ell^*}]$.
We claim that for any arc $\ova{e}$ of $P'$ there is total weight 
$1 \pm \eta^{.6}$ on wheels containing $\ova{e}$.
To see this, we compare the weight to $\hat{W}^1_{\ell^*}(\ova{e})$
as defined in section \ref{sec:reg}, 
which is $1 \pm 4\eps$ by Lemma \ref{degWc}
(as $p^1_{w,0}=(\ell^*-1) p^1_{w,\ell^*}$ 
and $p^1_*=(\ell^*-1) p^1_{\ell^*}$).
The actual weight on $\ova{e}$ differs from this estimate only due 
to wheels containing $\ova{e}$ that have another arc in $J'_1 \sd J_1$. 
There are at most $40\eta^{.7} n^{\ell^*-1}$ such wheels, 
each affecting the weight by at most $L^L n^{\ell^*-1}$, so the claim holds.
Thus regularity holds with $\dD=\eta^{.6}$ and $\oO=L^{-L}$.

It remains to show that $J'_1$ satisfies the 
extendability condition of Theorem \ref{decompL}.
Consider any disjoint $A,B \sub V$ and $C \sub W$
each of size $\le h$, where $h = 2^{50 (\ell^*)^3}$.
By Lemma \ref{deg}.iii,
for $c \in \{0,\ell^*\}$ we have
$$
|N^+_{J^0_1}(A) \cap N^-_{J^0_1}(B) \cap N^-_{J^c_1}(C)|
= |N_G^+(A) \cap N_G^-(B)| (p_*^1)^{|A|} (p_*^1)^{|B|}
\prod_{w \in C} p^1_{w,c} \pm 4sn^{3/4}
> (L^{-5} \aA )^{2h} n,
$$
by typicality of $G$.
Also, by Lemma \ref{deg}.iv
(with $S=\es$ and $S'=A \cup B$) we have
$|N^+_{J^0_1}(A) \cap N^+_{J^{\ell^*}_1}(B) \cap W| 
\ge 2^{-2s} L^{-7h} |W|$, say.
The perturbation from $J_1$ to $J'_1$ affects these estimates 
by at most $6h\eta^{.7}n < \eta^{.6}n$, so $J'_1$ satisfies 
extendability with $\oO=L^{-L}$ as above.
Now Theorem \ref{decompL} applies to give 
a $\wL$-decomposition of $J'_1$,
which completes the proof.
\end{proof}

Our second exact decomposition result concerns the path factors 
with prescribed ends required for Case $K$.
We recall that each $F^1_w$ consists of cycles of length $\ge K$
and at most one path of of length $\ge K$
with $|F^1_w| - n/2 \in [0,2K]$,
and that $(Y^1_w)^-$ and $(Y^1_w)^+$ are the sets of
startpoints and successors of intervals in $\mc{Y}^1_w$.
We also recall from Lemma \ref{lem:int} that for each $x \in V$,
letting $t^\pm_1(x) = |\{w: x \in (Y^1_w)^\pm\}|$,
we have $t^+_1(x)=t^-_1(x)=t_1$.
After embedding $F^2_w$, and a greedy embedding connecting
the paths to $(Y^1_w)^-$ and $(Y^1_w)^+$, we will need 
path factors $G^1_w$ as follows.

\begin{lemma} \label{exactK}
Suppose in Setting \ref{set} and Case $K$ that 
$Z_w$ is disjoint from $Y^1_w \cup (Y^1_w)^+$
and $8|Y^1_w| = n-|Z_w|-|(Y^1_w)^+|$ for all $w \in W$,
and $d_{G'_1}^\pm(x)=|W|-t_1-|Z(x)|$ for all $x \in V$.
Then $G'_1$ can be partitioned into graphs $(G^1_w: w \in W)$,
such that each $G^1_w$ is a vertex-disjoint union of oriented paths
with $V(G^1_w) = V\setminus Z_w$,
where for each $[a,b] \in \mc{Y}^1_w$ there
is an $ab^+$-path of length $8d(a,b)$.
\end{lemma}

\begin{proof}
We will show that there is 
a perturbation $P$ of $J_1$ such that
each $N^-_P(w) = V \sm Z_w$ and $P[V]$
corresponds to $G'_1$ under twisting, and
a set $E$ of arc-disjoint copies of $\wK$ in $P$,
such that Theorem \ref{decompK} applies to give 
a $\wK$-decomposition of $P' := P \sm \bigcup E$.
This will suffice, by taking each $G^1_w$ to consist of the
union of the oriented $8$-paths that correspond under twisting to
the rim $8$-cycles of the copies of $\wK$ containing $w$.

We construct $P$ by starting with $P=J_1$ and applying
a series of modifications as follows. First we delete all arcs
of $P[V]$ corresponding to arcs of $G_1 \sm G'_1$ and add 
arcs of colour $0$
corresponding to arcs of $G'_1 \sm G_1$.
Similarly, we delete all arcs $\ova{vw} \in P[V,W]$ 
with $v \in N^-_{J_1}(w) \cap Z_w$ and add arcs $\ova{vw}$
of colour $0$ for each 
$v \in V \sm (Z_w \cup (Y^1_w)^+ \cup  N^-_{J_1}(w))$.
We also replace any $\ova{xy}$ of colour $K$ with $d(x,y)<3d$
by an arc $\ova{xy}^+$ of colour $0$; 
this affects at most $6d$ arcs at each vertex.
Now $P[V]$ corresponds to $G'_1$ under twisting,
each $N^-_P(w) = V \sm (Z_w \cup (Y^1_w)^+)$ and
$P$ is a $2\eta^{.9}$-perturbation of $J_1$.

We note that $P$ now satisfies the divisibility condition
$d^-_P(w) = 8|Y^1_w| = 8d^-_{P^K}(w)$, and for each $v \in V$
that $d^+_P(v,W) = |W|-t_1-|Z(x)| = d_P(v,V)/2$, 
so $|P[V,W]| = |P[V]|$. We continue to modify $P$ to obtain
$|P^0[V,W]|=|P^0[V]|$ and $|P^K[V,W]|=|P^K[V]|$.
To do so, we will recolour arcs of $P[V]$ according to a 
greedy algorithm, where if $|P^0[V]|>|P^0[V,W]|$ we replace 
some $\ova{xy} \in P^0[V]$ by $\ova{xy}^- \in P^K[V]$,
or if $|P^0[V]|<|P^0[V,W]|$ we replace 
some $\ova{xy} \in P^K[V]$ by $\ova{xy}^+ \in P^0[V]$.
This preserves $P[V]$ corresponding to $G'_1$ under twisting
and $|P[V]| = |P[V,W]|$, so if we ensure $|P^0[V,W]|=|P^0[V]|$, 
we will also have $|P^K[V,W]|=|P^K[V]|$.
During the greedy algorithm, we choose the arc to recolour 
arbitrarily, subject to avoiding the set $S$ of vertices
at which we have recoloured more than $\eta^{.8} n/2$ arcs.
The total number of recoloured arcs is at most
$||P[V,W]|-|P[V]|| \le ||J_1[V,W]|-|J_1[V]|| + 2\eta^{.9}n^2
< 3\eta^{.9}n^2$ (by Lemma \ref{deg}), so $|S|<12\eta^{.1}n$.
Thus the algorithm can be completed,
giving $P$ that is an $\eta^{.8}$-perturbation of $J_1$
with $|P^0[V,W]|=|P^0[V]|$ and $|P^K[V,W]|=|P^K[V]|$.

We will continue modifying $P[V]$ until it satisfies the 
remaining degree divisibility conditions for each $v \in V$,
i.e.\ $d^+_P(v,V) = d^-_P(v,V) = d^+_P(v,W)$
and $d^-_{P^K}(v,V) =  d^+_{P^K}(v,W)$.
To do so, we will reduce to $0$ 
the imbalance $\DD' = \sum_{v \in V} \DD'(v)$ 
with each $\DD'(v) = |d^+_{P^K}(v,V)-d^+_{P^K}(v,W)|
+ |d^-_{P^K}(v,V)-d^+_{P^K}(v,W)|
$. We do not attempt 
to control any $d^\pm_{P^0}(v,V)$, but nevertheless
the divisibility conditions will be satisfied when $\DD'=0$.
To see this, note that if $\DD'=0$ then clearly all
$d^+_{P^K}(v,V)=d^-_{P^K}(v,V)=d^+_{P^K}(v,W)$,
so it remains to
show that $d_P^-(v,V)=d_P^+(v,V)=d_P^+(v,W)$.
Here we recall the discussion in section \ref{sec:alg}
relating the choice of intervals to degree divisibility,
where (setting $H=G'_1$ and $J=P$) we noted that
$d_{G'_1}^+(v) = d_P^+(v,V)$ and
$d_{G_1'}^-(v) = d_P^-(v,V) + \DD(v)$, with
$\DD(v) = d^-_{P^K}(v^-,V) - d^-_{P^K}(v,V)
= d^+_{P^K}(v^-,W) - d^+_{P^K}(v,W)$.
By our choice of intervals
all $d^+_{P^K}(v,W)$ are equal to $t_1$,
so $\DD(v)=0$ and $d_P^\pm(v,V) = d_{G'_1}^\pm(v)
= |W|-t_1-|Z(x)| = d^+_P(v,W)$, as required.

We have two types of reduction according to the two types
of term in the definition of $\DD'(v)$:
\begin{enumerate}
\item If $\sum_v |d^-_{P^K}(v,V)-d^+_{P^K}(v,W)| > 0$
then we can choose $x,y$ in $V$ with
$d^-_{P^K}(x,V) > d^+_{P^K}(x,W)$ and
$d^-_{P^K}(y,V) < d^+_{P^K}(y,W)$.
We will find $z \in V$ such that
$\ova{zx} \in P^K$, $\ova{zy}^+ \in P^0$
and replace these arcs by 
$\ova{zx}^+ \in P^0$, $\ova{zy} \in P^K$.
\item If $\sum_v |d^+_{P^K}(v,V)-d^+_{P^K}(v,W)| > 0$
then we can choose $x,y$ in $V$ with
$d^+_{P^K}(x,V) > d^+_{P^K}(x,W)$ and
$d^+_{P^K}(y,V) < d^+_{P^K}(y,W)$.
We will find $z \in V$ such that
$\ova{xz} \in P^K$, $\ova{yz}^+ \in P^0$
and replace these arcs by 
$\ova{yz} \in P^K$, $\ova{xz}^+ \in P^0$.
\end{enumerate}

\begin{center}
\includegraphics{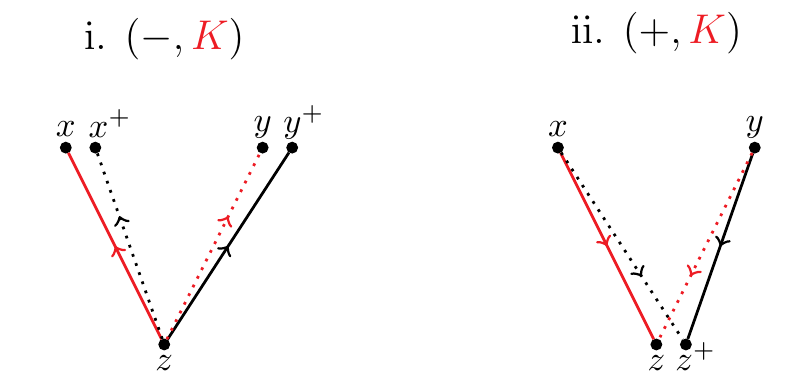}
\end{center}

Each of these operations preserves $P[V]$
corresponding to $G'_1$ under twisting and
reduces $\DD'$. 

To reduce $\DD'$ to $0$ we apply a greedy algorithm
where in each step we apply one of the above operations.
We do not allow $z$ with $d(x,z)<3d+2$ or $d(y,z)<3d+2$
(to avoid creating close arcs in colour $K$)
or $z$ in the set $S'$ of vertices
that have played the role of $z$ at $\eta^{.7}n/2$
previous steps. The total number of steps is at most 
$2\eta^{.8}n^2$, so $|S'| < 4\eta^{.1} n$.
To estimate the number of choices for $z$ at each step,
we apply Lemma \ref{deg}.iii to
$|N^-_{J^{K'}_1}(x^+) \cap N^-_{J^0_1}(y^+)|$
for operation (i),
$|N^+_{J^{K'}_1}(x) \cap N^+_{J^0_1}(y)|$ to find $z^+$ for (ii). 
By typicality of $G$ this gives 
at least $\aA^2 n/9$ choices,
of which at most $5\eta^{.1} n$ are forbidden 
by lying in $S$ or too close to $x$ or $y$, 
or due to requiring an arc of $J_1 \sm P$,
so some choice always exists.
Thus the algorithm can be completed, giving $P$ 
that is an $\eta^{.7}$-perturbation of $J_1$,
satisfies the divisibility conditions, and has 
$P[V]$ corresponding to $G'_1$ under twisting.

Next we construct $E$ as a set of arc-disjoint
copies of $\wK$ that cover all $\ova{xy} \in P[V]$
with $d(x,y)<3d$. Note that all such $\ova{xy}$ have colour $0$.
We apply a greedy algorithm, where in each step that we consider 
some $\ova{xy}$ we choose a copy of $\wK$ that is arc-disjoint
from all previous choices and does not use any vertex in the
set $S$ of vertices that have been used $.1d^2$ times.
Then $|S|.1d^2 < 27dn$, so this forbids
at most $270n^7/d$ choices of $\wK$.
By Lemma \ref{degWK} we have 
$\hat{W}^1_K(\ova{xy}) > 2^{-2s-1} p^1_K/p^1_* > 2^{-3s}$,
so the number of choices is at least
$2^{-3s} \min_{w \in W} W^2_{w,K}/p^2_{w,K}n
> 2^{-4s} n^7$, say. Thus there is always some choice 
that is not forbidden, so the algorithm can be completed.
We note that $\bigcup E$ has maximum degree 
at most $d^2$ by definition of $S$, so $P' := P \sm \bigcup E$ 
is a $2\eta^{.7}$-perturbation of $J_1$.
Furthermore, $P'$ satisfies the divisibility conditions,
as $P$ does and so does each $\wK$ in $E$.

Next we consider the regularity condition of Theorem \ref{decompK}.
To each $3d$-separated copy of $\wK$ in $P'$ with hub $w$ 
we assign weight $p^1_{w,K} n/ W^1_{w,K} = 
(\aA p^1_K (\aA p^1_* p^1_{w,0} n)^7 )^{-1}$,
which lies in $[n^{-7}, L n^{-7}]$.
We claim that for any arc $\ova{e}$ of $P'$ there is total weight 
$1 \pm \eta^{.6}$ on wheels containing $\ova{e}$.
To see this, we compare the weight to $\hat{W}^1_K(\ova{e})$
as defined in section \ref{sec:reg}, 
which is $1 \pm 4\eps$ by Lemma \ref{degWK}
(as $\ova{e}$ is $3d$-separated,
$p^1_{w,0}=7p^1_{w,K}$ and $p^1_*=7p^1_K$).
The actual weight on $\ova{e}$ differs from this estimate only due 
to wheels containing $\ova{e}$ that have another arc in $P' \sd J_1$. 
There are at most $40\eta^{.7} n^7$ such wheels, 
each affecting the weight by at most $Ln^{-7}$, so the claim holds.
Thus regularity holds with $\dD=\eta^{.6}$ and $\oO=L^{-1}$.

It remains to show that $P'$ satisfies the 
extendability condition of Theorem \ref{decompK}.
Consider any disjoint $A,B \sub V$ and $L \sub W$
each of size $\le h$ and $a, b, \ell \in \{0,K\}$.
By Lemma \ref{deg}.iii we have
$|N^+_{J_1^a}(A) \cap N^-_{J_1^b}(B) 
\cap N^-_{J_1^\ell}(L)| 
= |N_G^+(A) \cap N_G^-(B)| (p_1^a)^{|A|} (p_1^b)^{|B|}
\prod_{w \in L} p^1_{w,\ell} \pm 4sn^{3/4}
> (10^{-3} \aA )^{2h} n$, say.
Also, if $(A,B)$ is $3d$-separated then
by Lemma \ref{deg}.iv we have
$|N^+_{J_1^0}(A) \cap N^+_{J_1^K}(B) \cap W| 
\ge 2^{-2s+10h} |W|$, say.
The perturbation from $J_1$ to $P'$ affects these estimates 
by at most $6h\eta^{.7}n < \eta^{.6}n$, so $P'$ satisfies 
extendability with $\oO=L^{-1}$ as above.
Now Theorem \ref{decompK} applies to give 
a $\wK$-decomposition of $P'$,
which completes the proof.
\end{proof}

\section{The proof} \label{sec:pf}

This section contains the proof of our main theorem.
We give the reduction to cases in the first subsection
and then the proof for both cases in the second subsection.

\subsection{Reduction to cases} \label{sec:red}

In this subsection we formalise the reduction to cases
discussed in section \ref{sec:over}. For Theorem \ref{main},
we are given an $(\eps,t)$-typical $\aA n$-regular digraph $G$ 
on $n$ vertices, where $n^{-1} \ll \eps \ll t^{-1} \ll \aA$,
and we need to decompose $G$ into some given
family $\mc{F}$ of $\aA n$ oriented one-factors on $n$ vertices.
We prove Theorem \ref{main} assuming that it holds in the 
following cases with $t^{-1} \ll K^{-1} \ll \aA$:

Case $K$: each $F \in \mc{F}$ has at least $n/2$ vertices 
in cycles of length at least $K$,  

Case $\ell$ for all $\ell \in [3,K-1]$: 
each $F \in \mc{F}$ has $\ge K^{-3} n$ 
cycles of length $\ell$.

We will divide into subproblems via
the following partitioning lemma.

\begin{lemma} \label{typ:split}
Let $n^{-1} \ll \eps \ll t^{-1} \ll \aA_0$.
Suppose $G$ is an $(\eps,t)$-typical $\aA n$-regular digraph 
on $n$ vertices and $\aA = \sum_{i \in I} \aA_i$
with each $\aA_i > \aA_0$. Then $G$ can be decomposed
into digraphs $(G_i: i \in I)$ on $V(G)$ such that 
each $G_i$ is $(2\eps,t)$-typical and $\aA_i n$-regular.
\end{lemma}

\begin{proof}
We start by considering a random partition of $G$
into graphs $(G'_i: i \in I)$ where for each arc $\ova{e}$
independently we have $\mb{P}(\ova{e} \in G'_i)=\aA_i/\aA$.
We claim that whp each $G'_i$ is $(1.1\eps,t)$-typical.
Indeed, this holds by Chernoff bounds, 
as $\mb{E}d(G'_i) = \aA_i d(G)/\aA$ for each $i$, 
so whp $d(G'_i) = \aA_i \pm n^{-.4}$ (say),
and for any set $S=S_- \cup S_+$ 
of at most $t$ vertices, by typicality of $G$
we have $\mb{E} |N_{G'_i}^-(S_-) \cap N_{G'_i}^+(S_+)| 
= (\aA_i/\aA)^{|S|} |N_G^-(S_-) \cap N_G^+(S_+)|
= ((1  \pm \eps)d(G) \aA_i/\aA)^{|S|} n$,
so whp $|N_{G'_i}^-(S_-) \cap N_{G'_i}^+(S_+)| 
= ((1  \pm 1.1\eps) d(G'_i))^{|S|} n$,

Now we modify the partition to obtain $(G_i: i \in I)$,
by a greedy algorithm starting from all $G_i=G'_i$.
First we ensure that all $|G_i| = \aA_i n^2$.
At any step, if this does not hold then some
$|G_i| > \aA_i n^2$ and $|G_j| < \aA_j n^2$.
We move an arc from $G_i$ to $G_j$, arbitrarily
subject to not moving more than $n^{.7}$ arcs
at any vertex. We move at most $n^{1.6}$ arcs,
so at most $2n^{.9}$ vertices become forbidden 
during this algorithm. Hence the algorithm can 
be completed to ensure that all $|G_i| = \aA_i n^2$.
Each $|N_{G'_i}^-(S_-) \cap N_{G'_i}^+(S_+)|$ 
changes by at most $tn^{.7}$,
so each $G_i$ is now $(1.2\eps,t)$-typical.

Let $\wt{G_i}$ be the undirected graph of $G_i$
(which could have parallel edges).
We will continue to modify the partition until each $\wt{G_i}$ 
is $2\aA_i n$-regular, maintaining all $|G_i|=\aA_i n^2$.
At each step we reduce the imbalance
$\sum_{i,x} |d_{\wt{G_i}}(x)-2\aA_i n|$. 
If some $\wt{G_i}$ is not $2\aA_i n$-regular we have 
some $d_{\wt{G_i}}(x) > 2\aA_i n$ and $d_{\wt{G_i}}(y) < 2\aA_i n$.
Considering the total degree of $x$,
there is some $j$ with $d_{\wt{G_j}}(x) < 2\aA_j n$.
We will choose some $z$ with $xz \in \wt{G_i}$ and $yz \in \wt{G_j}$,
then move $xz$ to $\wt{G_j}$ and $yz$ to $\wt{G_i}$,
thus reducing the imbalance by at least $2$.
We will not choose $z$ in the set $L$ of vertices that
have played the role of $z$ at $n^{.8}$ previous steps.
We had all $d_{\wt{G_i}}(x) = 2(\aA_i n \pm n^{.7})$
after the first algorithm, so this algorithm will have
at most $2n^{1.7}$ steps, giving $|L| < n^{.9}$.
By typicality, there are at least $3\aA_i \aA_j n$ 
choices of $z$, of which at most $2n^{.9}$ are forbidden
by $L$ or requiring an edge that has been moved,
so the algorithm to make each $\wt{G_i}$ be $2\aA_i n$-regular
can be completed. Each $|N^-_{G_i}(S_-) \cap N^+_{G_i}(S_+)|$ 
changes by at most 
$tn^{.8}$, so each $G_i$ is now $(1.1\eps,t)$-typical.

We will continue to modify the partition until each $G_i$ 
is $\aA_i n$-regular, maintaining all $d_{\wt{G_i}}(x)=2\aA_i n$.
At each step we reduce the imbalance
$\sum_{i,x} |d_{G_i}^+(x)-\aA_i n|$
(if it is $0$ then since total degrees $d_{\wt{G_i}}(x)$
are correct, $G_i$ is regular).
If it is not $0$ we have 
some $d_{G_i}^+(x) > \aA_i n$ and $d_{G_i}^+(y) < \aA_i n$.
Again there is some $j$ with $d_{G_j}^+(x) < \aA_j n$ and we 
choose some $z$ with $\ova{xz} \in G_i$ and $\ova{yz} \in G_j$,
then move $\ova{xz}$ to $G_j$ and $\ova{yz}$ to $G_i$,
avoiding vertices $z$ which have played this role
at $n^{.9}$ previous steps.
By typicality we can find such $z$ at every step and
complete the algorithm. 
Each $|N_{G_i}^-(S_-) \cap N_{G_i}^+(S_+)|$ 
changes by at most 
$tn^{.9}$, so each $G_i$ is now $(2\eps,t)$-typical.
\end{proof}

Factors of a type that is too rare 
will be embedded greedily via the following lemma.

\begin{lemma} \label{typ:greedy}
Let $n^{-1} \ll \eps \ll t^{-1} \ll \aA$.
Suppose $G$ is an $(\eps,t)$-typical $\aA n$-regular digraph 
on $n$ vertices and $\mc{F}$ is a family of at most $\eps n$ oriented
one-factors. Then we can remove from $G$ a copy of each $F \in \mc{F}$
to leave a $(\sqrt{\eps},t)$-typical $(\aA n-|\mc{F}|)$-regular graph.
\end{lemma}

\begin{proof}
We embed the one-factors one by one. At each step,
the remaining graph $G'$ is obtained from $G$ by deleting
a graph that is regular of degree at most $2\eps n$,
so is $(\sqrt{\eps},t)$-typical. It is a standard argument
(which we omit) using the blow-up lemma 
of Koml\'os, S\'ark\"ozy and Szemer\'edi \cite{KSS} 
to show that any one-factor can be embedded in $G'$,
so the process can be completed.
\end{proof}

Now we prove Theorem \ref{main}
assuming that it holds in the above cases.
We introduce new parameters $\aA_1, \aA_2, M_1', M_1, M_2, M_3$ 
with $\eps \ll t^{-1} \ll M_3^{-1} \ll \aA_2 \ll M_2^{-1} 
\ll \aA_1 \ll (M_1')^{-1} \ll M_1^{-1} \ll \aA$.
For $\ell \in [3,M_2]$ let $\mc{F}_\ell$ consist of
all factors $F \in \mc{F}$ such that
$F$ has $\ge M_2^{-3} n$ cycles of length $\ell$
but $< M_2^{-3} n$ cycles of each smaller length.
Let $\mc{F}_2$ consist of all remaining factors in $\mc{F}$.
Note that each $F \in \mc{F}_2$ has fewer than $n/M_2$
vertices in cycles of length less than $M_2$,
so at least $(M_2-1)n/M_2$ in cycles of length at least $M_2$.
Let $B$ be the set of $\ell \in [3,M_2]$
such that $|\mc{F}_\ell| < \aA_2 n$.
Then for $\ell \in I' := [3,M_2] \sm B$ we have
$\bB_\ell := n^{-1} |\mc{F}_\ell| \ge \aA_2$. Also, 
writing $\mc{F}_B = \bigcup_{\ell \in B} \mc{F}_\ell$,
we have $\bB_B := n^{-1} |\mc{F}_B| < M_2\aA_2 < \sqrt{\aA_2}$.

Let $\mc{F}_1$ be the set of $F$ in $\mc{F}$ with
at least $n/2$ vertices in cycles of length $>M_1$.
We first consider the case 
$\eta := n^{-1}|\mc{F}_1| \ge \aA/2$.
Let $B^1 = B \cap [3,M_1]$,
$\mc{F}_{B^1} = \bigcup_{\ell \in B^1} \mc{F}_\ell$,
and $\bB_{B^1} := n^{-1} |\mc{F}_{B^1}| < \bB_B < \sqrt{\aA_2}$.
We apply Lemma \ref{typ:split} with
$I = (I' \cap [3,M_1]) \cup \{1\}$, letting
$\aA_\ell = \bB_\ell$ for all $\ell \in I' \cap [3,M_1]$
and $\aA_1 = \eta + \bB_{B^1}$, thus decomposing $G$ into 
$(2\eps,t)$-typical $\aA_i n$-regular digraphs $G_i$ on $V(G)$.
For each $\ell \in I' \cap [3,M_1]$ we decompose $G_\ell$ into
$\mc{F}_\ell$ by Case $\ell$ of Theorem \ref{main}, 
where in place of the parameters
$n^{-1} \ll \eps \ll t^{-1} \ll K^{-1} \ll \aA$ we use
$n^{-1} \ll 2\eps \ll t^{-1} \ll M_3^{-1} \ll \aA_2$.
For $G_1$, we first embed $\mc{F}_{B^1}$ via
Lemma \ref{typ:greedy}, leaving an $\eta n$-regular 
digraph $G'_1$ that is $(\eps',t)$-typical 
with $\aA_2 \ll \eps' \ll t{}^{-1} \ll M_2^{-1}$.
We then conclude the proof of this case by decomposing $G'_1$
into $\mc{F}_1$ by Case $K$ of Theorem \ref{main},
where in place of the parameters
$n^{-1} \ll \eps \ll t^{-1} \ll K^{-1} \ll \aA$ we use
$n^{-1} \ll \eps' \ll t{}^{-1} \ll M_1^{-1} \ll \eta$.
 
It remains to consider the case $\eta < \aA/2$.
Here there are at least $\aA n/2$ factors $F \in \mc{F}$
with at least $n/2$ vertices in cycles of length $\le M_1$,
so we can fix $\ell^* \in [M_1] \cap I'$ 
with $\bB_{\ell^*} > \aA/2M_1$. 
We consider two subcases according to $\bB_2 := n^{-1}|\mc{F}_2|$.

Suppose first that $\bB_2 < \aA_1 n$.
We apply Lemma \ref{typ:split} with $I = I'$, letting
$\aA_\ell = \bB_\ell$ for all $\ell \in I \sm \{\ell^*\}$
and $\aA_{\ell^*} = \bB_{\ell^*} + \bB_{B^1} + \bB_2$.
For each $\ell \in I \sm \{\ell^*\}$ we decompose $G_\ell$
into $\mc{F}_\ell$ by Case $\ell$ of Theorem \ref{main},
where (as before) in place of the parameters
$n^{-1} \ll \eps \ll t^{-1} \ll K^{-1} \ll \aA$ we use
$n^{-1} \ll 2\eps \ll t^{-1} \ll M_3^{-1} \ll \aA_2$.
For $G_{\ell^*}$ we first embed $\mc{F}_B \cup \mc{F}_2$
by Lemma \ref{typ:greedy}, leaving a $\bB_{\ell^*} n$-regular 
digraph $G'_{\ell^*}$ that is $(\eps',t)$-typical 
with $\aA_1 \ll \eps' \ll t{}^{-1} \ll M_1^{-1}$.
We then complete the decomposition by decomposing $G'_{\ell^*}$
into $\mc{F}_{\ell^*}$ by Case $\ell^*$ of Theorem \ref{main},
where in place of the parameters
$n^{-1} \ll \eps \ll t^{-1} \ll K^{-1} \ll \aA$ we use
$n^{-1} \ll \eps' \ll t{}^{-1} \ll (M_1')^{-1} \ll \bB_{\ell^*}$.

It remains to consider the subcase $\bB_2 \ge \aA_1 n$.
We apply Lemma \ref{typ:split} with $I = I' \cup \{2\}$, letting
$\aA_\ell = \bB_\ell$ for all $\ell \in I \sm \{\ell^*\}$
and $\aA_{\ell^*} = \bB_{\ell^*} + \bB_{B^1}$.
The same argument as in the first subcase applies to 
decompose $G_\ell$ into $\mc{F}_\ell$
for all $\ell \in I' \sm \{\ell^*\}$,
and also to embed $\mc{F}_B$ in $G_{\ell^*}$ 
by Lemma \ref{typ:greedy} and decompose 
the leave $G'_{\ell^*}$ into $\mc{F}_{\ell^*}$.
We complete the proof of this case, 
and so of the entire reduction, by decomposing 
$G_2$ into $\mc{F}_2$ by Case $K$ of Theorem \ref{main},
where in place of  the parameters
$n^{-1} \ll \eps \ll t^{-1} \ll K^{-1} \ll \aA$ we use
$n^{-1} \ll 2\eps \ll t^{-1} \ll M_2^{-1} \ll \bB_2$.

\subsection{Proof of Theorem \ref{main}}

We are now ready to prove our main theorem.
We are given an $(\eps,t)$-typical $\aA n$-regular digraph $G$ 
on $n$ vertices, where $n^{-1} \ll \eps \ll t^{-1} \ll \aA$,
and we need to decompose $G$ into some given
family $\mc{F}$ of $\aA n$ oriented one-factors on $n$ vertices.
By the reductions in section \ref{sec:red}, 
we can assume that we are in one of the 
following cases with $t^{-1} \ll M^{-1} \ll \aA$:

Case $K$: each $F \in \mc{F}$ has at least $n/2$ vertices 
in cycles of length at least $M$,  

Case $\ell^*$ with $\ell^* \in [3,M-1]$: each $F \in \mc{F}$ 
has $\ge M^{-3} n$ cycles of length $\ell^*$.

Here the parameters of section \ref{sec:red} are renamed:
$\ell$ is now $\ell^*$ so that 
`$\ell$' is free to denote generic cycle lengths;
$K$ is now $M$, as we want $K$ to take different 
values in each case: we introduce $M'$ with 
$t^{-1} \ll M'{}^{-1} \ll M^{-1}$ and define
\[ K = \left\{ \begin{array}{ll}
 M  & \text{ in Case } K, \\
 M' & \text{ in Case } \ell^*.
 \end{array} \right. \]
We define a parameter $L$ by $L=M$ in Case $\ell^*$
(so $\ell^* \ll L \ll K$), or as a new parameter
with $K^{-1} \ll L^{-1} \ll \aA$ in Case $K$.
We use these parameters to apply the algorithm 
of section \ref{sec:alg} as in~(\ref{hierarchy}), so we can apply the 
conclusions of the lemmas in 
sections \ref{sec:int} to \ref{sec:exact}.

We recall that each factor $F_w$
is partitioned as $F^1_w \cup F^2_w$,
where $F^1_w$ either consists of exactly
$L^{-3} n$ cycles of length $\ell^*$ in Case $\ell^*$,
or in Case $K$ we have $|F^1_w| - n/2 \in [0,2K]$
and $F^1_w$ consists of cycles of length $\ge K$
and at most one path of length $\ge K$
(and then $F^2_w = F_w \sm F^1_w$).

By Lemma \ref{lem:approx},
there are arc-disjoint digraphs $G^2_w \sub G_2$ for $w \in W$,
where each $G^2_w$ is a copy of some valid $F'_w \sub F^2_w$
with $V(G^2_w) \sub N^-_{J_2}(w)$, such that 
\begin{enumerate}
\item $G_2^- = G_2 \sm \bigcup_{w \in W} G^2_w$
has maximum degree at most $5d^{-1/3}n$,
\item the digraph $J_2^-$
obtained from $J_2[V,W]$ by deleting all $\ova{xw}$ 
with $x \in V(G^2_w)$ has maximum degree at most $5d^{-1/3}n$,
\item any $x \in V$ has degree $1$ in $F'_w$
for at most $n/\sqrt{d}$ choices of $w$.
\end{enumerate}
(Recall that `valid' means that $F'_w$
does not have any independent arcs, 
and if $F^2_w$ contains a path then $F'_w$
contains the arcs incident to each of its ends.)

Note that (ii) implies for each $w \in W$ that
$|F'_w| \ge |N^-_{J_2}(w)| - 5d^{-1/3}n
> p^2_w n - 6d^{-1/3}n$ (by Lemma \ref{deg}),
so as $p^2_w n = (1-\eta)|F^2_w| + n^{.8}$
we have $|F^2_w \sm F'_w| < \eta n$.

Next we will embed oriented graphs 
$R_w = (F^2_w \sm F'_w) \cup L_w$ for $w \in W$,
where $L_w \sub F^1_w$ is defined as follows.
In Case $\ell^*$ we let each $L_w$ consist of 
$2\eta L^{-3} n$ cycles of length $\ell^*$.
In Case $K$ we partition each $F^1_w$ as $\mc{P}_w \cup L_w$,
where $\mc{P}_w$ is a valid vertex-disjoint union of paths, 
such that for each $[a,b] \in \mc{Y}^1_w$ we have 
an oriented path $P^{ab}_w$ in $\mc{P}_w$ of length $8d(a,b)$
(which we will embed as an $ab^+$-path).
To see that such a partition exists, we apply the same
argument as at the end of the proof of Lemma \ref{lem:approx}.
We consider a greedy algorithm, where at each step that we consider 
some path $P^{ab}_w$ we delete a path of length $8d(a,b)+4$
from $F^1_w$, which we allocate as $P^{ab}_w$ surrounded on both
sides of paths of length $2$ that we add to $L_w$.
As $|\mc{Y}^1_w| < n/2d_{2s+1} = (2s)^{2s}n/2d$ we thus
allocate $< (2s)^{2s}n/d$ edges to $L_w$.
Suppose for contradiction that the algorithm gets stuck,
trying to embed some path $P$ in some remainder $Q_w$.
Then all components of $Q_w$ have size $\le 8d+5$.
All components of $F^1_w$ have size $\ge K$,
so $|Q_w| \le (8d+5)|F^1_w|/K < 5dn/K$.
However, we also have 
$|Q_w| \ge |F^1_w|-|Y^1_w|-|L_w| \ge \eta n/3$, 
as $|F^1_w| \ge n/2$ and 
$|Y^1_w| = (1-\eta)n/2 \pm 2n^{3/4}$ by Lemma \ref{lem:int}.
This is a contradiction, so the algorithm finds a partition 
$F^1_w = \mc{P}_w \cup L_w$ with $\mc{P}_w$ valid.
We note that each $|R_w| < 2\eta n$.

Now we apply a greedy algorithm to construct arc-disjoint
embeddings $(\phi_w(R_w): w \in W)$ in $G_1$.
At each step we choose some $\phi_w(x) \in N^-_{J^1}(w)$
(which is disjoint from $G^2_w \sub N^-_{J_2}(w)$).
We require $\phi_w(x)$ to be an outneighbour of some 
previously embedded $\phi_w(x_1)$ or both an outneighbour of 
$\phi_w(x_1)$ and an inneighbour of $\phi_w(x_2)$
for some previously embedded images; 
the latter occurs when we finish a cycle or a path
(the image under $\phi_w$ of the ends of the paths in $R_w$ 
have already been prescribed:
they are either images of endpoints of paths in $F'_w$
or startpoints / successors of intervals in $\mc{Y}^1_w$).
We also require $\phi_w(x)$ to be distinct from all
previously embedded $\phi_w(x_1)$ and not to lie 
in the set $S$ of vertices that are already in the image 
of $\phi_{w'}$ for at least $\eta^{.9} n/2$ choices of $w'$. 
As $\eta^{.9} n |S|/2 \le \sum_{w \in W} |R_w| < 2\eta n^2$
we have $|S| < 4\eta^{.1} n$. 
To see that it is possible to choose $\phi_w(x)$, 
first note for any $v,v'$ in $V$ and $w \in W$ that 
$|N_{G_1}^+(v) \cap N_{G_1}^-(v') \cap  N^-_{J^1}(w)| > \aA^2 n/3$,
by Lemma \ref{deg}.iii and typicality of $G$.
At most $|R_w|+|S| < 5\eta^{.1} n$ choices of $\phi_w(x)$ 
are forbidden due to using $S$ or some previously embedded
$\phi_w(x_1)$. Also, by definition of $S$, we have used
at most $\eta^{.9} n$ arcs at each of $v$ and $v'$ for
other embeddings $\phi_{w'}$, so this forbids 
at most $2\eta^{.9} n$ choices of $\phi_w(x)$.
Thus the algorithm never gets stuck, so we can construct 
$(\phi_w(R_w): w \in W)$ as required.

Let $G'_1 = G \sm \bigcup_{w \in W} (G^2_w \cup R_w)$.
For each $w \in W$ let $Z_w$ be the set 
of vertices of in- and outdegree $1$ in $G^2_w \cup R_w$.
We claim that $G'_1$ and $Z_w$ satisfy Setting \ref{set}.
To see this, first note that by definition of $S$ above 
each $|N_{G_1}^\pm(x) \sm N_{G'_1}^\pm(x)| < \eta^{.9} n/2$.
As $d_{G_2^-}^\pm(x) < 5d^{-1/3}n$ by (i) above and
(by Lemma~\ref{deg}) $d_G^\pm(x)-d_{G_1}^\pm(x)-d_{G_2}^\pm(x) 
< (1-p_1-p_2)d_G^\pm(x) + n^{.6} < 2\eta n$
we have $|N_{G_1}^\pm(x) \sd N_{G'_1}^\pm(x)| < \eta^{.9} n$,
so $G'_1$ is an $\eta^{.9}$-perturbation of $G_1$.
Also, as $|N^-_{J_2}(w) \sm F'_w| \le 5d^{-1/3}n$,
$|R_w| < 2\eta n$ and $|V \sm N^-_J(w)| < 2\eta n$
(the last by Lemma \ref{deg}) we have
$|Z_w \sd (V \sm N^-_{J^1}(w))| < 5\eta n$, as claimed.

In Case $\ell^*$, every vertex 
has equal in- and outdegrees $0$ or $1$ in $G^2_w \cup R_w$ 
(it is a vertex-disjoint union of cycles) so 
$d_{G'_1}^\pm(x)=|W|-|Z(x)|$ for all $x \in V$
and $\ell^*$ divides $n-|Z_w|$ for all $w \in W$.
Thus Lemma \ref{exactL} applies to partition $G'_1$ 
into graphs $(G^1_w: w \in W)$, where each $G^1_w$ 
is a $C_{\ell^*}$-factor with $V(G^1_w) = V \sm Z_w$,
thus completing the proof of this case.

In Case $K$, a vertex $x$ has indegree (respectively outdegree) 
$1$ in $G^2_w \cup R_w$
exactly when $x \in (Y^1_w)^-$ 
(respectively $(Y^1_w)^+$),
for which there are each $t_1$ choices of $w$,
so $d_{G'_1}^\pm(x)=|W|-t_1-|Z(x)|$ for all $x \in V$.
By construction, $Z_w$ is disjoint from $(Y^1_w)^- \cup (Y^1_w)^+$,
and the total length of paths required in the remaining path 
factor problem satisfies
$8|Y^1_w| = n-|Z_w|-|(Y^1_w)^+|$ for all $w \in W$.
Thus Lemma \ref{exactK} applies to partition $G'_1$ 
into graphs $(G^1_w: w \in W)$,
such that each $G^1_w$ is a vertex-disjoint union of oriented paths
with $V(G^1_w) = V\setminus Z_w$,
where for each $[a,b] \in \mc{Y}^1_w$ there
is an $ab^+$-path of length $8d(a,b)$.
This completes the proof of this case,
and so of Theorem \ref{main}.

\section{Concluding remarks} \label{sec:con}

As mentioned in the introduction,
our solution to the generalised Oberwolfach Problem
is more general than the result of \cite{GJKKO}
in three respects: it applies to any typical graph
(theirs is for almost complete graphs)
and to any collection of two-factors
(they need some fixed $F$ to occur $\OO(n)$ times),
and it applies also to directed graphs.
Although there are some common elements 
in both of our approaches (using \cite{K2} 
for the exact step and some form of twisting),
the more general nature of our result 
reflects a greater flexibility in our approach
that has further applications.
One such application is our recent proof \cite{KSringel}
that every quasirandom graph with
$n$ vertices and $rn$ edges can be
decomposed into $n$ copies of any fixed
tree with $r$ edges.
The case of the complete graph solves 
Ringel's tree-packing conjecture~\cite{ringel} 
(solved independently via different
methods by Montgomery, Pokrovskiy 
and Sudakov~\cite{MPS3}).

A natural open problem raised in \cite{GJKKO}
is whether the generalised Oberwolfach problem
can be further generalised to decompositions of $K_n$
into any family of regular graphs of bounded degree
(where the total of the degrees is $n-1$).

\end{document}